\documentclass[11pt,twoside]{aarticle}

\titlefont{\usefont{T1}{yvtjw}{b}{it}\LARGE}
\authorfont{\usefont{T1}{qcs}{m}{n}\fontsize{14}{16}}
\parttitlefont{\usefont{OT1}{phv}{l}{it}\fontsize{28}{33}}
\sectiontitlefont{\usefont{T1}{yvtjw}{b}{it}\fontsize{15}{20}}
\sectionheadfont{\usefont{OT1}{pbk}{l}{it}\fontsize{12}{15}}
\sectionmarkfont{\usefont{OT1}{pbk}{l}{sl}\fontsize{11}{15}}
\subsectiontitlefont{\usefont{T1}{yvtjw}{b}{it}\fontsize{13}{18}}
\subsectionheadfont{\usefont{OT1}{pbk}{l}{it}\fontsize{12}{15}}
\subsectionmarkfont{\usefont{OT1}{pbk}{l}{sl}\fontsize{11}{15}}
\subsubsectiontitlefont{\usefont{T1}{pbk}{l}{it}}
\paragraphtitlefont{\normalfont\normalsize}
\subparagraphtitlefont{\normalfont\normalsize}
\headfont{\usefont{OT1}{pbk}{l}{n}\fontsize{13}{15}}
\datefont{\usefont{T1}{qcs}{m}{n}}
\bibliographyname{References}
\usepackage[square]{natbib}
\usepackage[svgnames]{xcolor}
\definecolor{DarkBlue}{rgb}{0, 0, 0.545}
\usepackage[colorlinks,urlcolor=DarkBlue,citecolor=DarkBlue,linkcolor=DarkBlue,anchorcolor=DarkBlue]{hyperref}
\usepackage{amsmath,nicefrac,wasysym,booktabs,layout,ntheorem,stmaryrd,enumitem,todonotes,amssymb,qpalatin,todonotes}
\usepackage[british]{babel}
\usepackage[series=b]{libgreek}
\def\mybf#1{\textsf{#1}}
\usepackage{tikz}
\usetikzlibrary{matrix,arrows,calc,positioning}

\tikzset{every scope/.style={>=angle 60,thick}}
\newcommand{\comment}[1]{}

\newcommand\Figref[1]{Figure \ref{#1}\ifthenelse{\value{page}=\pageref{#1}}{}{ on page \pageref{#1}}}
\newcommand\figref[1]{figure \ref{#1}\ifthenelse{\value{page}=\pageref{#1}}{}{ on page \pageref{#1}}}

\def\bt{\begin{center}\begin{tikzpicture}\matrix[matrix of math nodes,column sep=1cm, row sep=1cm]}
\def\bth{\begin{center}\begin{tikzpicture}\matrix[matrix of math nodes,column sep=.5cm, row sep=1cm]}
\def\et{\end{tikzpicture}\end{center}}
\def\btm#1#2{\begin{center}\begin{tikzpicture}\matrix[matrix of math nodes,column sep=#1, row sep=#2}
\def\bsc{\begin{scope}}
\def\esc{\end{scope}}

\def\|#1{\,\begin{tikzpicture}[baseline]\draw[-,thin](0,-3pt)--(0,6pt);\node[anchor=west] at (-1pt,-2pt) {$_{#1}\!\!\!$}; \end{tikzpicture}\,}

\def\hyphen{\,\text{-}\,}

\def\coloneq{\mathrel{\mathop:}=}

\def\cotensor{\oblong}
\def\otimes{\varotimes}

\def\im{\mathbf{im}}

\def\ker{\mathsf{ker}}
\def\coker{\mathsf{coker}}

\def\int{\mathsf{int}}

\def\LatT1{\mathsf{Lat_{T1}}}
\newcommand{\id}{\operatorname{id}}
\newcommand{\can}{\operatorname{\it can}}

\def\bQ{\mathbb{Q}}

\def\bZ{\mathbb{Z}}

\def\Alg{\mathit{Alg}}

\def\Sub{\mathsf{Sub}}

\def\Id{\mathsf{Id}}

\def\coId{\mathsf{coId}}
\def\coid#1{\coId_{\mathit{#1}}}

\def\inf{\mathit{inf}}
\def\Quot{\mathsf{Quot}}
\def\qquot{\mathsf{Quot_{\mathit{gen}}}}
\def\qsub{\mathsf{Sub_{\mathit{gen}}}}
\def\qid{\mathsf{Id_{\mathit{gen}}}}
\def\id{\mathit{id}}
\def\can{\mathit{can}}

\theoremstyle{break}
\theorembodyfont{\itshape}
\theoremheaderfont{\usefont{T1}{ptm}{bx}{it}\normalsize}
\newtheorem{theorem}{Theorem}[section] 

\newtheorem{proposition}[theorem]{Proposition}

\newenvironment{proof}{\par\noindent
    \textbf{Proof:}}{\nopagebreak[4]\hfill$\blacksquare$\medskip\\}
\newenvironment{proofof}[1]{\par\noindent \textbf{Proof of {#1}:}}{\nopagebreak[4]\hfill$\blacksquare$\medskip\\}

\newtheorem{definition}[theorem]{Definition}

\newtheorem{lemma}[theorem]{Lemma \usefont{T1}{ptm}{m}{sl}}
\newtheorem{corollary}[theorem]{Corollary}

\theoremstyle{nonumberbreak}
\theorembodyfont{\itshape}
\theoremheaderfont{\usefont{T1}{fxlf}{b}{ic}\normalsize}
\renewtheorem{theorem*}{Theorem}
\renewtheorem{mtheorem*}{Main Theorem}
\renewtheorem{proposition*}{Proposition}
\renewtheorem{construction*}{Construction}
\renewtheorem{definition*}{Definition}
\renewtheorem{observation*}{Observation}
\renewtheorem{lemma*}{Lemma}
\theoremstyle{nonumberplain}
\renewtheorem{corollary*}{Corollary}

\theorembodyfont{\normalfont}
\theoremheaderfont{\usefont{T1}{fxlf}{b}{ic}\normalsize}
\theoremstyle{plain}
\newtheorem{example}[theorem]{Example  \usefont{T1}{ptm}{m}{sl}}

\theoremstyle{nonumberplain}
\renewtheorem{example*}[theorem]{Example}
\renewtheorem{examples*}[theorem]{Examples}
\renewtheorem{exercise*}{Exercise}

\theorembodyfont{\slshape}
\theoremheaderfont{\usefont{T1}{fxlf}{b}{ic}\normalsize}
\theoremstyle{plain}
\newtheorem{remark}[theorem]{Remark  \usefont{T1}{ptm}{m}{sl}}

\theoremstyle{break}

\theoremstyle{nonumberbreak}
\renewtheorem{notation*}{Notation}

\theoremstyle{nonumberplain}
\renewtheorem{remark*}{Remark}
\renewtheorem{note*}{Note}
\renewtheorem{question*}{Question}

\theorembodyfont{\normalfont}
\theoremheaderfont{\usefont{T1}{fxlf}{b}{ic}\normalsize}
\def\ov{\overline}
\def\fill#1{\hbox to #1pt{\hfill}}

\def\galoisskip{.8mm}
\def\galoislength{7mm} 	
\def\shortgaloislength{5mm} 
\def\longgaloislength{12mm} 
\def\arrowlength{6mm}
\def\arrowhight{2pt}

\def\shortarrowlength{4.75mm}
\def\longarrowlength{10mm}

\def\horizontalskip{3pt}
\def\nodeposition{.5}		
\def\shortnodeposition{.3}	
\def\elmaptipheight{.9mm}      

\def\wmpr#1{\smash{\mathop{\hbox to 12pt{\rightarrowfill}}\limits^{#1}}}

\def\mpr#1{
\hspace{\horizontalskip}\begin{tikzpicture}[baseline=-1pt]
	\draw[->] (0,\arrowhight) -- node[above,pos=\nodeposition]{${\scriptstyle #1}$} (\arrowlength,\arrowhight);
\end{tikzpicture}
\hspace{\horizontalskip}}

\def\ir{
\hspace{\horizontalskip}\begin{tikzpicture}[baseline=-1pt]
	\draw[->] (0,\arrowhight) -- (\arrowlength,\arrowhight);
\end{tikzpicture}
\hspace{\horizontalskip}}

\def\smpr#1{
\hspace{\horizontalskip}\begin{tikzpicture}[baseline]
    \draw[->] (0,\arrowhight) -- node[above,pos=\shortnodeposition]{${\scriptstyle #1}$} (\shortarrowlength,\arrowhight);
\end{tikzpicture} \hspace{\horizontalskip}}

\def\sir{
\hspace{\horizontalskip}\begin{tikzpicture}[baseline=-1pt]
    \draw[->] (0,\arrowhight) -- (\shortarrowlength,\arrowhight);
\end{tikzpicture} \hspace{\horizontalskip}}

\def\eir{
\hspace{\horizontalskip}\begin{tikzpicture}[baseline=-1pt]
    \draw[->>] (0,\arrowhight) -- (\arrowlength,\arrowhight);
\end{tikzpicture} \hspace{\horizontalskip}}

\def\lmpr#1{
\hspace{\horizontalskip}\begin{tikzpicture}[baseline=-1pt]
    \draw[->] (0,\arrowhight) -- node[above,pos=\nodeposition]{${\scriptstyle #1}$} (\longarrowlength,\arrowhight);
\end{tikzpicture} \hspace{\horizontalskip}}

\def\smrp#1{\hspace{\horizontalskip}\smash{\mathop{\hbox to 12pt{\rightarrowfill}}\limits_{#1}}\hspace{\horizontalskip}} 

\def\smrp#1{\hspace{\horizontalskip}\smash{\mathop{\hbox to 12pt{\rightarrowfill}}\limits_{#1}}\hspace{\horizontalskip}}

\def\elmap#1{
\hspace{\horizontalskip}\begin{tikzpicture}[baseline=-1pt]
\draw[->] (0,\arrowhight) -- node[above,pos=\nodeposition]{${\scriptstyle #1}$} (\arrowlength,\arrowhight);
\draw[-] ($(0,\arrowhight)+(0,\elmaptipheight)$) -- ($(0,\arrowhight)+(0,-\elmaptipheight)$);
\end{tikzpicture}
\hspace{\horizontalskip}}

\def\eli{
\hspace{\horizontalskip}\begin{tikzpicture}[baseline=-1pt]
\draw[->] (0,\arrowhight) --  (\arrowlength,\arrowhight);
\draw[-] ($(0,\arrowhight)+(0,\elmaptipheight)$) -- ($(0,\arrowhight)+(0,-\elmaptipheight)$);
\end{tikzpicture}
\hspace{\horizontalskip}}

\def\selmap#1{
\hspace{\horizontalskip}\begin{tikzpicture}[baseline=-1pt]
\draw[->] (0,\arrowhight) -- node[above,pos=\nodeposition]{${\scriptstyle #1}$} (\shortarrowlength,\arrowhight);
\draw[-] ($(0,\arrowhight)+(0,\elmaptipheight)$) -- ($(0,\arrowhight)+(0,-\elmaptipheight)$);
\end{tikzpicture}
\hspace{\horizontalskip}}

\def\ellmap#1{
\hspace{\horizontalskip}
\begin{tikzpicture}[baseline=-1pt]
\draw[<-] (0,\arrowhight) -- node[above,pos=\nodeposition]{${\scriptstyle #1}$} (\arrowlength,\arrowhight);
\draw[-] ($(\arrowlength,\arrowhight)+(0,\elmaptipheight)$) -- ($(\arrowlength,\arrowhight)+(0,-\elmaptipheight)$);
\end{tikzpicture}
\hspace{\horizontalskip}}

\def\galois#1#2{
\hspace{\horizontalskip}
\begin{tikzpicture}[baseline=-4pt]
\begin{scope}
\draw[->] (0,\galoisskip) -- node[above,pos=.5]{${\scriptstyle #1}$} (\galoislength,\galoisskip);
\draw[<-] (0,-\galoisskip) -- node[below,pos=.6]{${\scriptstyle #2}$} (\galoislength,-\galoisskip);
\end{scope}
\end{tikzpicture}
\hspace{\horizontalskip}}

\def\sgalois#1#2{
\hspace{\horizontalskip}
\begin{tikzpicture}[baseline=-4pt]
\begin{scope}
\draw[->] (0,\galoisskip) -- node[above,pos=.5]{${\scriptstyle #1}$} (\shortgaloislength,\galoisskip);
\draw[<-] (0,-\galoisskip) -- node[below,pos=.6]{${\scriptstyle #2}$} (\shortgaloislength,-\galoisskip);
\end{scope}
\end{tikzpicture}
\hspace{\horizontalskip}}

\def\sGalois{
\hspace{\horizontalskip}
\begin{tikzpicture}[baseline=-4pt]
\begin{scope}
\draw[->] (0,\galoisskip) --  (\shortgaloislength,\galoisskip);
\draw[<-] (0,-\galoisskip) --  (\shortgaloislength,-\galoisskip);
\end{scope}
\end{tikzpicture}
\hspace{\horizontalskip}}

\def\lgalois#1#2{
\hspace{\horizontalskip}
\begin{tikzpicture}[baseline=-4pt]
\begin{scope}
\draw[->] (0,\galoisskip) -- node[above,pos=.5]{${\scriptstyle #1}$} (\longgaloislength,\galoisskip);
\draw[<-] (0,-\galoisskip) -- node[below,pos=.5]{${\scriptstyle #2}$} (\longgaloislength,-\galoisskip);
\end{scope}
\end{tikzpicture}
\hspace{\horizontalskip}}

\def\arrowhight{2.5pt}
\def\horizontalskip{1pt}

\begin{document}
\titlefont{\usefont{T1}{pbk}{m}{it}\fontsize{17}{24}}
\title{\mbox{Galois Theory for H-extensions and H-coextensions}}
\authori{Dorota Marciniak}
\authoriaddress{National Institute of Telecommunications, Warsaw, Poland}
\authoriemail{dorofia@gmail.com}
\authorii{Marcin Szamotulski\footnote{Second address: National Institute of
	Telecommunications, Warsaw, Poland.}}
\authoriiaddress{IST, Lisbon, Portugal}
\authoriiemail{mszamot@gmail.com}
\maketitle
\thispagestyle{empty}
\begin{abstract}
We show that there exists a Galois correspondence between subalgebras of an H-comodule algebra~A over a base ring R and generalised quotients of a Hopf algebra H. We also show that Q-Galois subextensions are closed elements of the constructed Galois connection. Then we consider the theory of coextensions of H-module coalgebras. We construct Galois theory for them and we prove that H-Galois coextensions are closed. We apply the obtained results to the Hopf algebra itself and we show a simple proof that there is a bijection correspondence between right ideal coideals of H and its left coideal subalgebras when H is finite dimensional. Furthermore we formulate necessary and sufficient conditions when the Galois correspondence is a bijection for arbitrary Hopf algebras. We also present new conditions for closedness of subalgebras and generalised quotients when A is a crossed product.
\end{abstract}	

\section{Introduction}
Hopf--Galois extensions have roots in the approach
of~\cite{sc-dh-ar:galois-theory} who generalised the \emph{classical Galois
    Theory} for field extensions to commutative rings. In the following paper
\cite{sc-ms:hopf-algebras-and-galois-theory} extended these ideas to coactions
of Hopf algebras on \emph{commutative algebras} over rings. The general
definition of a Hopf-Galois extension was first introduced by
\cite{hk-mt:hopf-algebras-and-galois-extensions}. Under the assumption that
\(H\) is finite dimensional their definition is equivalent to the now days
standard
\begin{definition}
	An $H$-extension $A/A^{co\,H}$ is called \mybf{$H$-Hopf-Galois
	    extension} (\mybf{$H$-Galois extension}, for short) if the
	\mybf{canonical map} of right $H$-comodules and left $A$-modules:
	\begin{equation}\label{eq:canonical-map}
	    \can:A\otimes_{A^{co\,H}} A\sir A\otimes H,\ a\otimes b\eli{}ab_{(0)}\otimes b_{(1)}
	\end{equation}
	is an \emph{isomorphism},
	\(A^{co\,H}:=\{a\in A: a_{(0)}\otimes a_{(1)}=a\otimes 1_H\}\). 
\end{definition}

A breakthrough was made by extending the results of
\cite{sc-ms:hopf-algebras-and-galois-theory} to noncommutative setting by
\citeauthor{fo-yz:gal-cor-hopf-galois}. They construct Galois correspondence
for Hopf-Galois extensions which has particularly good properties for field
extensions (see~\citep[Thm~4.4 and Cor.~4.6]{fo-yz:gal-cor-hopf-galois}).
Van~Oystaeyen and Zhang introduce a remarkable construction of an
\emph{associated Hopf algebra} to an $H$-extension $A/A^{co\,H}$, where $A$ as
well as $H$ are supposed to be commutative
(\cite[Sec.~3]{fo-yz:gal-cor-hopf-galois}, for noncommutative
generalisation see:~\cite{ps:hopf-bigalois,ps:gal-cor-hopf-bigal}). We will
denote this Hopf algebra by $L(H,A)$.
\citet[Prop.~3.2]{ps:gal-cor-hopf-bigal} generalises the van Oystaeyen and
Zhang correspondence (see also~\cite[Thm~6.4]{ps:hopf-bigalois}) to Galois
connection between generalised quotients of the associated Hopf algebra
\(L(H,A)\) (i.e. quotients by right ideal coideals) and subextensions of
a faithfully flat \(H\)-Hopf Galois extension of the base ring. In this work
we construct a Galois correspondence without the assumption that the
coinvariants subalgebra is commutative and we also drop the Hopf--Galois
assumption (Theorem~\ref{thm:existence}). Instead of Hopf theoretic approach
of van Oystaeyen, Zhang and Schauenburg we propose to look from lattice
theoretic perspective. Using an existence theorem for Galois connections we
show that if the comodule algebra \(A\) is flat over \(R\) and the functor
\(A\otimes_R-\) preserves infinite intersections then there exists a Galois
correspondence between subalgebras of \(A\) and generalised quotients of the
Hopf algebra \(H\). It turns out that such modules are exactly the
Mittag--Leffler modules (Corollary~\ref{cor:mittag-leffler}). We consider
modules with intersection property in
Section~\ref{sec:modules_with_int_property}, where we also give examples of
flat and faithfully flat modules which fail to have it.  Then we discuss
Galois closedness of generalised quotients and subalgebras. We show that if
a generalised quotient \(Q\) is such that \(A/A^{co\,Q}\) is \(Q\)-Galois then
it is necessarily closed assuming that \(A/A^{co\,H}\) has epimorphic
canonical map (Corollary~\ref{cor:Q-Galois_closed}). Later we prove that this
is also a necessary condition for Galois closedness if \(A=H\) or, more
generally, if \(A/A^{co\,H}\) is a crossed product, \(H\) is flat and
\(A^{co\,H}\) is a flat Mittag--Leffler \(R\)-module
(Theorem~\ref{thm:cleft-case}). We also consider the dual case: of
\(H\)-module coalgebras, which later gives us a simple proof of bijective
correspondence between generalised quotients and left ideal subalgebras
of~\(H\) if it is finite dimensional (Theorem~\ref{thm:newTakeuchi}). This
Takeuchi correspondence, dropping the assumptions of faithfully (co)flatness
in~\cite[Thm.~3.10]{ps:gal-cor-hopf-bigal}, was proved
by~\cite{ss:projectivity-over-comodule-algebras}, who showed that finite
dimensional Hopf algebra is free over any its left coideal subalgebra. Our
proof avoids using this result. We also characterise closed elements of this
Galois correspondence in general case (Theorem~\ref{thm:closed-of-qquot}). As
we already mentioned, we show that a generalised quotient \(Q\) is closed if
and only if \(H/H^{co\,Q}\) is a \(Q\)-Galois extension. Furthermore, we show
that a left coideal subalgebra~\(K\) is closed if and only if \(H\sir H/K^+H\)
is a \(K\)-Galois coextension (see Definition~\ref{defi:coGalois}). This gives
an answer to the question when the bijective correspondence between
generalised quotients over which~\(H\) is faithfully coflat and coideal
subalgebra over which~\(H\) is faithfully flat holds without (co)flatness
assumptions. In the last section we extend the characterisation of closed
subalgebras and closed generalised quotients to crossed products.

\section{Preliminaries}\label{subsec:basics}
A \mybf{partially ordered set}, or \mybf{poset} for short, is a set $P$
together with a \textsf{reflexive}, \textsf{transitive} and
\textsf{anti-symmetric} relation $\leq$. The dual poset to a poset $P$ we will
denote by $P^\mathit{op}$. If the partial order of $P$ is denoted by $\leq$
then the partial order $\leq^\mathit{op}$ of $P^\mathit{op}$ is defined by
\(p_1\leq^\mathit{op}p_2\Leftrightarrow p_2\leq p_1\). Infima in $P$ will be
denoted by $\bigwedge$, i.e.  $\inf_{i\in I}p_i=\bigwedge_{i\in I}p_i$, for
$p_i\in P$.  We will write \(\bigvee\) for suprema in \(P\). A poset which has
all finite infima and suprema is called a \mybf{lattice}.  It is called
\mybf{complete} if arbitrary infima and suprema exist.

\begin{definition}\label{defi:Galois_connection} Let \(P\) and \(Q\) be two
    posets. Then a \mybf{Galois connection} is a pair \((F,G)\) of
    antimonotonic maps: \(F:P\sGalois Q:G \text{ such that }\forall_{p\in
	    P,q\in Q}\quad GFp\geq p\text{ and }FG(q)\geq q\). An element \(p\)
    of \(P\) (respectively \(q\in Q\)) is called \mybf{closed} if and only if
    \(GFp=p\) (\(FGq=q\)). The set of closed elements of \(P\) (\(Q\)) we will
    denote by \(\ov P\) (\(\ov Q\) respectively).
\end{definition}
\begin{proposition}[\cite{bd-hp:introduction-to-lattices}]\label{prop:properties-of-adjunction}
	Let \((F,G)\) be a Galois connection between the poset $P$ and $Q$. Then:
	\begin{enumerate}[topsep=0pt,noitemsep]
		\item[(1)] $\ov P=G(Q)$ and $\ov Q=F(P)$,
		\item[(2)] The restrictions $F|_{\ov P}$ and $G|_{\ov Q}$ are
		    \textsf{inverse bijections} of $\ov P$ and $\ov Q$ (\(\ov
			P\) and \(\ov Q\) are largest such that \(F\) and
		    \(G\) restricts to inverse bijections).
		\item[(3)] The map $F$ is \textsf{unique} in the sense that
		    there exists only one Galois connection of the form
		    \((F,G)\), in a similar way $G$ is \textsf{unique}. 
		\item[(4)] The map $F$ is \textsf{mono} (\textsf{onto}) if and
		    only if the map $G$ is \textsf{onto} (\textsf{mono}). 
		\item[(5)] If one of the two maps \(F,G\) is an
		    \textsf{isomorphism} then the second is its
		    \textsf{inverse}.
	\end{enumerate}
	If \(P\) and \(Q\) are complete lattices then so are the posets of
	closed elements.
\end{proposition}
Let us note that if $(F,G)$ is a Galois connection between posets then \(F\) and \(G\)
reflects all existing suprema into infima.
\begin{theorem}\label{thm:existence-of-adjunction}
	Let \(P\) and \(Q\) be two posets. Let \(F:P\sir Q\) be an
	anti-monotonic map of posets. If \(P\) is complete then there
	\textsf{exists Galois connection} \((F,G)\) if and only if
	\textsf{\(F\) reflects all suprema}.
\end{theorem}
For the proof we refer to~\cite[Prop.~7.34]{bd-hp:introduction-to-lattices}.

All algebras, coalgebras, Hopf algebras, etc. if not otherwise stated, are
assumed to be over a commutative ring~\(R\). 
The unadorned tensor product $\otimes$ will denote the tensor product over the
base ring $R$.
We refer to or~\cite{tb-rw:corings-and-comodules} for the basic definitions.
Let us recall that a \mybf{coideal} of an \(R\)-coalgebra \(C\) is kernel of
a coalgebra epimorphism with source $C$.
	For a coalgebra $C$  over a ring \(R\) the set of all coideals of the
	coalgebra $C$ forms a \textsf{complete lattice} with inclusion as
	the order relation. It will be denoted by $\coId(C)$.
The complete poset of left coideals we let denote by \(\coId_l(C)\).
A \mybf{Hopf ideal} of a Hopf algebra $H$ is a kernel of an epimorphism of
Hopf algebras with source $H$. 
	The set of Hopf ideals form a complete lattice which will be denoted
	by $\Id_\mathit{Hopf}(H)$. The dual lattice we will denote by
	$\Quot(H)$.

\begin{definition}
    A \mybf{generalised quotient} \(Q\) of a Hopf algebra \(H\) is a quotient
    by a right ideal coideal. The {poset of generalised quotients} will be
    denoted by \(\qquot(H)\). The order relation of \(\qquot(H)\) we will
    denote by \(\succcurlyeq\).

\noindent A \mybf{generalised subalgebra}~\(K\) of a Hopf algebra~\(H\) is
a left coideal subalgebra. The poset of generalised subalgebras will be denoted
by~\(\qsub(H)\).
\end{definition}
The poset \(\qquot(H)\) is dually isomorphic to the \emph{poset of right
    ideals coideals} of \(H\), which will be denoted as \(\qid(H)\). We define
only right version of generalised quotients and left version of generalised
subalgebras since we will deal only with right \(H\)-comodules.
\begin{proposition}
    Let \(H\) be a Hopf algebra. Then the poset $\qquot(H)$ is a complete
    lattice.
\end{proposition}
\begin{proof}
    Note that the supremum in \(\qid(H)\) is given by sum of submodules. It
    follows that \(\qid(H)\) is a complete upper lattice, and thus it is
    a complete lattice, since in any complete lattice we have:
    \(x\wedge y=\mathop\bigvee\limits_{z\leq x\text{ and }z\leq y}z\).
\end{proof}
The infimum in \(\qid(H)\) is given by the formula:
\[I\wedge J=\mathop+\limits_{\substack{K\subseteq I\cap J\\ K\in\qid(H)}}K\]
where \(I,J\in\qid(H)\). Furthermore, let us note that if \(H\) is finite
dimensional (over a field) then the lattice of generalised ideals of a Hopf is
both algebraic and dually algebraic. 

	An $R$-algebra $A$ is an \mybf{$H$-comodule algebra} if  $A$ is
	a (coassociative and counital) $H$-como\-dule with structure maps:
	\(\delta_A:A\sir A\otimes H\) and $\id_A\otimes\epsilon:A\otimes H\sir
	A$, which are algebra homomorphisms (with the usual algebra structure
	on the tensor product).
We denote the \mybf{subalgebra of coinvariants} by $A^{co\,H}\coloneq\{a\in
    A:\delta_A(a)=a\otimes 1_H\}$. 
For an algebra $A$ the \emph{poset of all subalgebras} will be denoted by
$\Sub_{\Alg}(A)$ and $\Sub_{\Alg}(A/B):=\{S\in\Sub_{\Alg}(A):B\subseteq S\}$.
Let $Q$ be a generalised quotient of $H$. Then $A$ is a $Q$-comodule with the
structure map $\delta_Q:=\id\otimes\pi_Q\circ\delta_A$, where $\pi_Q:H\sir Q$
is the projection. If $Q$ is a Hopf quotient then this coaction turns $A$ into
a $Q$-comodule algebra.

\section{Modules with intersection
    property}\label{sec:modules_with_int_property}
For a flat \(R\)-module \(M\) the tensor product functor \(M\otimes-\)
preserves all finite intersections
(see~\cite[40.16]{tb-rw:corings-and-comodules}). Furthermore, it is not hard
to show that tensoring with a flat module preserves all finite limits. In this
section we show that there is a large class of modules for which the tensor
product functor preserves arbitrary intersections. We will also construct
examples of flat and faithfully modules without this property.

Let $N'$ be a submodule of $N$, $i:N'\subseteq N$, and let $M$ be an
$R$-module. Then the \mybf{canonical image} of $M\otimes N'$ in $M\otimes N$
is the image of \(M\otimes N'\) under the map $\id_M\otimes i$. It will be
denoted by $\im(M\otimes N')$.
\begin{definition}
    Let \(M,N\) be an \(R\)-modules, and let \((N_\alpha)_{\alpha\in I}\) be
    a family of submodules of an \(R\)-module~\(N\).  We say that a module
    \(M\) has the \textbf{intersection property with respect to \(N\)} if the
    homomorphism:
       \[\im\bigl(M\otimes\bigl(\bigcap_{\alpha\in I}N_\alpha\bigr)\bigr)\ir \bigcap_{\alpha\in I}\im\bigl(M\otimes N_\alpha\bigr)\] 
    is an isomorphism for any family of submodules \((N_\alpha)_{\alpha\in
	    I}\). We say that \(M\) has the \textbf{intersection property} if
    the above condition holds for any \(R\)-module~\(N\).
\end{definition}
Note that if \(M\) is flat then it has the intersection property if and only
if the map \(M\otimes(\bigcap_{\alpha\in I}N_\alpha)\ir \bigcap_{\alpha\in
	I}(M\otimes N_\alpha)\) is an isomorphism. 

\begin{proposition}
    The intersection property is closed under direct sums.
\end{proposition}
\begin{proof}
    Let \(X=\oplus_{i\in I}X_i\), be a direct sum of modules with intersection
    property. We let \(\pi_i:\oplus_{i\in I}X_i\sir X_i\) be the canonical
    projection on \(i\)-th factor and \(s_i:X_i\sir\oplus_{i\in I}X_i\) be the
    canonical section. Let \((N_\alpha)_{\alpha\in J}\) be a family of
    submodules of an \(R\)-module \(N\). Then we have a split epimorphism
    \((\oplus_{i\in I}\pi_i)\otimes\id_N\) with a right inverse
    \((\oplus_{i\in I}s_i)\otimes\id_N\). Also for each \(\alpha\in J\) the
    map \((\oplus_{i\in I}\pi_i)\otimes\id_{N_\alpha}\) is a split epimorphism
    with right inverse \((\oplus_{i\in I}s_i)\otimes\id_{N_\alpha}\).
    Furthermore, we have a family of split epimorphisms, which sections are
    jointly surjective:
    \begin{center}
	\begin{tikzpicture}
	    \matrix[column sep=2cm]{
		\node (A) {\(\im\left((\oplus_{i\in I}X_i)\otimes N_\alpha\right)\)}; & \node (B) {\(\im(X_i\otimes N_\alpha)\)};\\
	    };
	    \draw[->>] (A) --node[below]{\(\pi_i\otimes\id_{N}\)} (B);
	    \draw[<-<] ($(A)+(1.4cm,3mm)$) .. controls +(8mm,5mm) and +(-8mm,5mm) ..node[above]{\(s_i\otimes\id_{N}\)} ($(B)+(-1.0cm,3mm)$);
	\end{tikzpicture}
    \end{center}
    They induce the following family of projections with jointly surjective
    sections:
    \begin{center}
	\begin{tikzpicture}
	    \matrix[column sep=2cm]{
		\node (A) {\(\bigcap_{\alpha\in J}\im\left((\oplus_{i\in I}X_i)\otimes N_\alpha\right)\)}; & \node (B) {\(\bigcap_{\alpha\in J}\im(X_i\otimes N_\alpha)\)};\\
	    };
	    \draw[->>] (A) --node[below]{\(\pi_i\otimes\id_{N}\)} (B);
	    \draw[<-<] ($(A)+(1.8cm,4mm)$) .. controls +(8mm,5mm) and +(-8mm,5mm) ..node[above]{\(s_i\otimes\id_{N}\)} ($(B)+(-1.4cm,4mm)$);
	\end{tikzpicture}
    \end{center}
    For this let \(x\in\bigcap_{\alpha\in J}\im\left((\oplus_{i\in
		I}X_i)\otimes N_\alpha\right)\). Then, for each \(\alpha\in
	J\), there exists \(y_\alpha\in\oplus_{i\in I}\im(X_i\otimes
	N_\alpha)\) such that \(\left(\oplus_{i\in
		I}s_i\otimes\id_N\right)(y_\alpha)=\sum_{i\in
	    I}s_i\otimes\id_{N}(y_\alpha)=x\). Since \(\oplus_{i\in
	    I}s_i\otimes\id_N\) is a monomorphism we get
    \(y=y_\alpha\in\bigcap_{\alpha\in J}\oplus_{i\in I}\im(X_i\otimes
	N_\alpha)\) for all \(\alpha\in J\).  It follows that:
    \[\bigcap\im\left((\oplus_{i\in I}X_i)\otimes N_\alpha\right)\lmpr{\oplus_{i\in I}\pi_i\otimes\id_N}\oplus_{i\in I}\bigcap_{\alpha\in J}\im(X_i\otimes N_\alpha)\] 
    is an isomorphism with inverse \(\oplus_{i\in I}s_i\otimes\id_N\). 
    Now the proposition will follow from the commutativity of the diagram:
    \begin{center}
	\hspace*{-1cm}{\hfill\begin{tikzpicture}
	    \matrix[column sep=1cm,row sep=1cm]{
	    \node (A0) {\(\im\left((\oplus_{i\in I}X_i)\otimes(\bigcap_{\alpha\in J}N_\alpha)\right)\)}; &                                                  & \node (B0) {\(\bigcap_{\alpha\in J}\im\left((\oplus_{i\in I}X_i)\otimes N_\alpha\right)\)};\\
	    \node (A1) {\(\im(\oplus_{i\in I}(X_i\otimes(\bigcap_{\alpha\in J}N_\alpha)))\)}; & \node (A2) {\(\oplus_{i\in I}\im(X_i\otimes(\bigcap_{\alpha\in J}N_\alpha))\)}; & \node (B2) {\(\oplus_{i\in I}\bigcap_{\alpha\in J}\im(X_i\otimes N_\alpha)\)};\\
	    };
	    \draw[->] (A0) -- (B0);
	    \draw[->] (A0) --node[left]{\(\simeq\)} (A1);
	    \draw[->] (A1) --node[below]{\(=\)} (A2);
	    \draw[->] (A2) --node[below]{\(\simeq\)} (B2);
	    \draw[->] (B2) --node[right]{\(\simeq\)} (B0);
	\end{tikzpicture}\nolinebreak[4]
	\hfill\refstepcounter{equation}\raisebox{12mm}{(\theequation)}\label{diag:proof-sums-and-intersection}}
    \end{center}
    We will go around this diagram from the top left corner to the top
    right one and prove that all the maps on the way are isomorphisms. The
    first map is an isomorphism since tensor product commutes with
    colimits. 
    Clearly the second map is an isomorphism as well. 
    The bottom right arrow in~\eqref{diag:proof-sums-and-intersection} is an
    isomorphism since all \(X_i\) (\(i\in I\)) have the intersection property
    and we already showed that the last homomorphism is an isomorphism.

\end{proof}
\begin{proposition}
    The intersection property is stable under taking direct summands.
\end{proposition}
\begin{proof}
    Let \(M\) be a direct summand in a module \(P\) which has intersection
    property. Let \(M'\) be the complement of \(M\) in \(P\). Then we have
    a chain of isomorphisms:
    \begin{align*}
	\im\Bigl(M\otimes\bigcap_{\alpha\in J}N_\alpha\Bigr)\oplus\im\Bigl(M'\otimes\bigcap_{\alpha\in J}N_\alpha\Bigr) & \simeq \im\Bigl((M\oplus M')\otimes(\bigcap_{\alpha\in J}N_\alpha)\Bigr) \\
						  & =\bigcap_{\alpha\in J}\im\Bigl((M\oplus M')\otimes N_\alpha\Bigr) \\
						  & \simeq\bigcap_{\alpha\in J}\im\left(M\otimes N_\alpha\oplus M'\otimes N_\alpha\right) \\
						  & \simeq\bigcap_{\alpha\in J}\im\left(M\otimes N_\alpha\right)\oplus\bigcap_{\alpha\in J}\im(M'\otimes N_\alpha)
    \end{align*}
    Since every isomorphism in above diagram commutes with projection onto the
    first and second factor the composition also does, and thus it is direct
    sum of the two natural maps: \(\im\Bigl(M\otimes\bigcap_{\alpha\in
	    J}N_\alpha\Bigr)\sir\bigcap_{\alpha\in J}\im\left(M\otimes
	    N_\alpha\right)\), \(\im\Bigl(M'\otimes\bigcap_{\alpha\in
	    J}N_\alpha\Bigr)\sir\bigcap_{\alpha\in J}\im\left(M'\otimes
	    N_\alpha\right)\). It follows that both maps are isomorphisms,
    hence both \(M\) and \(M'\) have the intersection property. 
\end{proof}
Since pure projective modules are direct summands in sums of finitely
presented modules and tensor product with finitely presented modules preserves
all limits we get that:
\begin{corollary}\label{cor:projective-modules}
    Every pure projective module has the intersection property.
\end{corollary}
\begin{proposition}
    Let \(M\) be a Mittag--Leffler \(R\)-module. Then \(M\) has the
    intersection property.
\end{proposition}
We wish to thank Christian Lomp for presenting us this result and pointing us
to~\cite{dh-jt:mittag-leffler}.  Note that for the case of flat
Mittag--Leffler modules the above result follows
from~\cite[Cor.~2.1.7]{mr-lg:platitude-et-projectivte}.
\begin{proof}
    Let \(N_\alpha\) for \(\alpha\in I\) be a family of submodule an
    \(R\)-module \(N\). Let us consider the following diagram:
    \begin{center}
	\begin{tikzpicture}
	    \matrix[column sep=1cm,row sep=7mm]{
		\node (A) {\(0\)}; & \node (B) {\(M\otimes\left(N/\bigcap_{\alpha\in I}N_\alpha\right)\)};          & \node (C) {\(M\otimes\left(\prod_{\alpha\in I}N/N_\alpha\right)\)}; \\
				   &                                             & \node (D) {\(\prod_{\alpha\in I}\bigl(M\otimes (N/N_\alpha)\bigr)\)}; \\
		\node (G) {\(0\)}; & \node (F) {\((M\otimes N)/\bigcap_{\alpha\in I}\im(M\otimes N_\alpha)\)}; & \node (E) {\(\prod_{\alpha\in I}(M\otimes N)/\im(M\otimes N_\alpha)\)}; \\
	    };
	    \draw[->] (A) -- (B);
	    \draw[->] (B) --node[above]{\(i\)} (C);
	    \draw[->] (C) --node[right]{\(f\)} (D);
	    \draw[->] (D) --node[right]{\(g\)} (E);
	    \draw[<-] (E) --node[below]{\(j\)} (F);
	    \draw[<-] (F) -- (G);
	    \draw[->,dashed] (B) --node[left]{\(G\)} (F);
	\end{tikzpicture}
    \end{center}
    Where \(i\) and \(j\) are the canonical embeddings:
    \(i\left(m\otimes(n+\bigcap_{\alpha\in
		I}N_\alpha)\right):=m\otimes(n+N_\alpha)_{\alpha\in I}\) and
    \(j\left((m\otimes n)+\bigcap_{\alpha\in I}\im\left(M\otimes
		N_\alpha\right)\right)=\left(m\otimes n+\im(M\otimes
	    N_\alpha)\right)_{\alpha\in I}\) for \(m\in M\) and \(n\in N\).
    While \(f\) sends \(m\otimes(n_\alpha+N_\alpha)_{\alpha\in I}\) to
    \(\left(m\otimes (n_\alpha+N_\alpha)\right)_{\alpha\in I}\) and \(g\) is
    the canonical isomorphism. Note that \(\im(gfi)\subseteq\im(j)\) and hence
    if \(M\) is Mittag--Leffler, then \(G:=gfi\) can be considered an
    embedding \(G:M\otimes\left(N/\bigcap_{\alpha\in
		I}N_\alpha\right)\sir(M\otimes N)/\bigcap_{\alpha\in
	    I}\im(M\otimes N_\alpha)\). Hence we get the exact diagram:
    \begin{center}
	\begin{tikzpicture}
	    \matrix[column sep=1cm,row sep=7mm]{
		    & \node (A1) {\(0\)};                    & \node (A2) {\(0\)};    & \node (A3) {\(0\)};                          & \\
\node (B0) {\(0\)}; & \node (B1) {\(\im\left(M\otimes\left(\bigcap_{\alpha\in I}N_\alpha\right)\right)\)}; & \node (B2) {\(M\otimes N\)}; & \node (B3) {\(M\otimes\left(N/\bigcap_{\alpha\in I}N_\alpha\right)\)};          & \node (B4) {\(0\)}; \\
\node (C0) {\(0\)}; & \node (C1) {\(\bigcap_{\alpha\in I}\im\left(M\otimes N_\alpha\right)\)};  & \node (C2) {\(M\otimes N\)}; & \node (C3) {\((M\otimes N)/\bigcap_{\alpha\in I}\im\left(M\otimes N_\alpha\right)\)}; & \node (C4) {\(0\)}; \\
		    & \node (D1) {\(\coker(H)\)};            & \node (D2) {\(0\)};    & \node (D3) {\(\coker(G)\)};                  & \\
	    };
	    \draw[->] (B0) -- (B1);
	    \draw[->] (B1) -- (B2);
	    \draw[->] (B2) -- (B3);
	    \draw[->] (B3) -- (B4);
	    \draw[->] (C0) -- (C1);
	    \draw[->] (C1) -- (C2);
	    \draw[->] (C2) -- (C3);
	    \draw[->] (C3) -- (C4);

	    \draw[->] (A1) -- (B1);
	    \draw[->] (A2) -- (B2);
	    \draw[->] (A3) -- (B3);

	    \draw[->] (B1) --node[left]{\(H\)} (C1);
	    \draw[->] (B2) --node[left]{\(=\)} (C2);
	    \draw[->] (B3) --node[right]{\(G\)} (C3);

	    \draw[->] (C1) -- (D1);
	    \draw[->] (C2) -- (D2);
	    \draw[->] (C3) -- (D3);
	\end{tikzpicture}
    \end{center}
    Where \(H\) is the canonical embedding of
    \(\im\left(M\otimes\left(\bigcap_{\alpha\in I}N_\alpha\right)\right)\)
    into \(\im\left(\bigcap_{\alpha\in I}\left(M\otimes
		N_\alpha\right)\right)\) as submodules of \(M\otimes N\). By
    the snake lemma we get the following short exact sequence: 
    \[0=\ker(G)\sir\coker(H)\sir 0\]
    Thus the monomorphism~\(H\) is onto.
\end{proof}
Since a module is flat Mittag--Leffler if and only if it is
\(\aleph_1\)-projective (see~\cite[Thm.~2.9]{dh-jt:mittag-leffler}) we get the
following 
\begin{corollary}
    Any \(\aleph_1\)-projective module has the intersection property.
\end{corollary}
Now using~\cite[Prop.~2.1.8]{mr-lg:platitude-et-projectivte} we obtain:
\begin{corollary}\label{cor:mittag-leffler}
    A flat module has the intersection property if and only if it satisfies
    the Mittag--Leffler condition.
\end{corollary}

\begin{example}\label{ex:flat_without_ip}\footnote{We wish to thank S.Papadakis
	for discussions which led to these examples.}
    Let \(p\) be a prime ideal of \(\bZ\). Then \(\bigcap_ip^i=\{0\}\). We let
    \(\bZ_p\) denote the ring of fractions of \(\bZ\)  with respect to \(p\).
    The \(\bZ\)-module \(\bZ_p\) is flat, and
    \(\bZ_p\otimes_\bZ\bigcap_ip^i=\{0\}\). From the other side
    \(\bigcap_i\bZ_p\otimes_\bZ p^i\cong\bZ_p\). By similar argument
    \(\bQ\otimes_\bZ-\) doesn't posses the intersection property, even though
    it is flat over~\(\bZ\). The problem is that, the intersection property is
    stable under arbitrary sums but not under cokernels. Now it is easy to
    construct a faithfully flat module which does not have the intersection
    property. The \(\bZ\)-modules \(\bZ\oplus\bZ_p\) and \(\bZ\oplus\bQ\) are
    the examples.
\end{example}
In the proof of Proposition~\ref{cor:projective-modules} we showed that the
intersection property is stable under split exact sequences. However, the
above examples show that the intersection property is not stable under pure
(exact) sequences~\cite[Def.~4.83]{tl-modules}, i.e. whenever \(0\sir M'\sir
    M\sir M''\sir 0\) is a pure exact sequence and \(M\) has the intersection
property then~\(M''\) might not have it. It is well known that if \(M''\) is
flat then \(M'\subseteq M\) is pure~\cite[Thm~4.85]{tl-modules},
hence~\ref{ex:flat_without_ip} is indeed a source of examples. However, we
can show the following proposition:
\begin{proposition}
    Let \(M'\) be a pure submodule of a module \(M\) with the intersection
    property. Then \(M'\) has the intersection property.
\end{proposition}
This property is shared by the classes of flat or Mittag--Leffler modules.
\begin{proof}
    We have commutative diagram with exact rows:
    \begin{center}
	\hspace*{-1cm}
	\begin{tikzpicture}[style={>=angle 60,thin},cross line/.style={preaction={draw=white, -, line width=5pt}}]
	    \node (B0) at (0,0) {\(0\)};
	    \node (B1) at (3.2,0) {\(\im\left(M\otimes\left(\bigcap_{\alpha\in I}N_\alpha\right)\right)\)};
	    \node (B2) at (7,0) {\(M\otimes N\)};
	    \node (B3) at (11,0) {\(M\otimes\left(N/\bigcap_{\alpha\in I}N_\alpha\right)\)};
	    \node (B4) at (14.5,0) {\(0\)};
	    \node (C0) at (0,-4) {\(0\)};
	    \node (C1) at (3.2,-4){\(\bigcap_{\alpha\in I}\im\left(M\otimes N_\alpha\right)\)};
	    \node (C2) at (7,-4) {\(M\otimes N\)};
	    \node (C3) at (11,-4){\((M\otimes N)/\bigcap_{\alpha\in I}\im\left(M\otimes N_\alpha\right)\)};
	    \node (C4) at (14.5,-4) {\(0\)};
	    \draw[->] (B0) -- (B1);
	    \draw[->] (B1) -- (B2);
	    \draw[->] (B2) --node[fill=white]{\(g\)} (B3);
	    \draw[->] (B3) -- (B4);
	    \draw[->] (C0) -- (C1);
	    \draw[->] (C1) -- (C2);
	    \draw[->] (C2) -- (C3);
	    \draw[->] (C3) -- (C4);

	    \draw[->] (B1) --node[left]{\(H\)}node[above,rotate=-90]{\(\simeq\)} (C1);
	    \draw[->] (B2) --node[above,rotate=-90]{\(=\)} (C2);
	    \draw[->] (B3) --node[right]{\(G\)} (C3);

	    \node[xshift=-2cm,yshift=-1.4cm] (D0) at (0,0) {\(0\)};
	    \node[xshift=-2cm,yshift=-1.4cm] (D1) at (3.2,0) {\(\im\left(M'\otimes\left(\bigcap_{\alpha\in I}N_\alpha\right)\right)\)};
	    \node[xshift=-2cm,yshift=-1.4cm] (D2) at (7,0) {\(M'\otimes N\)};
	    \node[xshift=-2cm,yshift=-1.4cm] (D3) at (11,0) {\(M'\otimes\left(N/\bigcap_{\alpha\in I}N_\alpha\right)\)};
	    \node[xshift=-2cm,yshift=-1.4cm] (D4) at (14.5,0) {\(0\)};
	    \node[xshift=-2cm,yshift=-1.4cm] (E0) at (0,-4) {\(0\)};
	    \node[xshift=-2cm,yshift=-1.4cm] (E1) at (3.2,-4){\(\bigcap_{\alpha\in I}\im\left(M'\otimes N_\alpha\right)\)};
	    \node[xshift=-2cm,yshift=-1.4cm] (E2) at (7,-4) {\(M'\otimes N\)};
	    \node[xshift=-2cm,yshift=-1.4cm] (E3) at (11,-4){\((M'\otimes N)/\bigcap_{\alpha\in I}\im\left(M'\otimes N_\alpha\right)\)};
	    \node[xshift=-2cm,yshift=-1.4cm] (E4) at (14.5,-4) {\(0\)};
	    \draw[->,cross line] (D0) -- (D1);
	    \draw[->,cross line] (D1) -- (D2);
	    \draw[->,cross line] (D2) --node[fill=white]{\(f\)} (D3);
	    \draw[->,cross line] (D3) -- (D4);
	    \draw[->,cross line] (E0) -- (E1);
	    \draw[->,cross line] (E1) -- (E2);
	    \draw[->,cross line] (E2) -- (E3);
	    \draw[->,cross line] (E3) -- (E4);

	    \draw[->,cross line] (D1) --node[left]{\(H'\)} (E1);
	    \draw[->,cross line] (D2) --node[above,rotate=-90]{\(=\)} (E2);
	    \draw[->] (D3) --node[right]{\(G'\)} (E3);

	    \draw[>->] (D1) -- (B1);
	    \draw[>->] (D2) -- (B2);
	    \draw[>->] (D3) -- (B3);

	    \draw[>->] (E1) -- (C1);
	    \draw[>->] (E2) -- (C2);
	    \draw[->] (E3) -- (C3);
	\end{tikzpicture}
    \end{center}
    It easily follows that \(H'\) is a monomorphism. Let
    \(x\in\bigcap_{\alpha\in I}\im\left(M'\otimes N_\alpha\right)\). To prove
    that \(x\) is in the image of \(H'\) it is enough to show that it goes to
    \(0\) under \(f\). Now since \(H\) is an isomorphism it goes to \(0\)
    under \(g\) and thus it belongs to the kernel of \(f\).
\end{proof}
\begin{theorem}
    Every flat \(R\)-module has the intersection property (or equivalently is
    Mittag--Leffler) if and only if for any exact sequence: 
    \[0\sir M'\sir M\sir M''\sir 0\] 
    with \(M'\), \(M\) projective, \(M''\) flat and any family of submodules
    \((N_\alpha)_{\alpha\in I}\) of an \(R\)-module \(N\) the sequence:
    \begin{equation}\label{eq:every_flat}
            0\sir \bigcap_\alpha \left(M'\otimes N_\alpha\right)\sir\bigcap_\alpha \left(M\otimes N_\alpha\right)\sir\bigcap_\alpha \left(M''\otimes N_\alpha\right)\sir 0
    \end{equation}
    is exact.
\end{theorem}
\begin{proof}
    Every flat module is a colimit of projective modules. Any colimit of
    projective modules can be computed as a cokernel of a map between
    projective modules, by~\cite[Thm~1, Chap.~V,
    \S2]{smc:categories-for-the-working-mathematician}.  So let \(M''\) be
    a flat module and let 
    \[\mathcal{E}:\ 0\sir M'\sir M\sir M''\sir 0\]
    be exact, where \(M'\) and \(M\) are projective modules, hence they
    satisfy the intersection property. The extension \(\mathcal{E}\) is pure,
    since \(M''\) is flat~\cite[Thm~4.85]{tl-modules}. Let
    \((N_\alpha)_{\alpha\in I}\) be a family of submodules of an \(R\)-module
    \(N\). We have a commutative diagram:
    \begin{center}
	\begin{tikzpicture}
	    \matrix[matrix, column sep=1cm, row sep=0.75cm]{
		\node(A) {\(0\)}; & \node(B) {\(M'\otimes\bigcap_\alpha N_\alpha\)};   & \node(C) {\(M\otimes\bigcap_\alpha N_\alpha\)};   & \node(D) {\(M''\otimes\bigcap_\alpha N_\alpha\)};   & \node(E) {\(0\)};\\
		\node(U) {\(0\)}; & \node(W) {\(\bigcap_\alpha\left(M'\otimes N_\alpha\right)\)}; & \node(X) {\(\bigcap_\alpha\left(M\otimes N_\alpha\right)\)}; & \node(Y) {\(\bigcap_\alpha\left(M''\otimes N_\alpha\right)\)}; & \node(Z) {\(0\)};\\
	    };
	    \draw[->] (A) -- (B);
	    \draw[->] (B) -- (C);
	    \draw[->] (C) -- (D);
	    \draw[->] (D) -- (E);
	    \draw[->] (U) -- (W);
	    \draw[->] (W) -- (X);
	    \draw[->] (X) -- (Y);
	    \draw[->] (Y) -- (Z);
	    \draw[->] (B) --node[left]{\rotatebox{90}{\(\cong\)}} (W);
	    \draw[->] (C) --node[left]{\rotatebox{90}{\(\cong\)}} (X);
	    \draw[->] (D) --node[right]{\(f\)} (Y);
	\end{tikzpicture}
    \end{center}
    The upper row is exact by purity of \(M'\subseteq M\), thus the lower row
    is exact if and only if the canonical map \(f\) is an isomorphism.
\end{proof}
The exactness of~\eqref{eq:every_flat} is rather difficult to obtain but it
might be useful once we know that every flat module is Mittag--Leffler. 

\section{Galois Theory for Hopf-Galois extensions}
Now we can prove existence of the Galois correspondence for comodule algebras.
\begin{theorem}\label{thm:existence}
    Let \(A/B\) be a \(H\)-comodule algebra over a ring \(R\) such that \(A\)
    is a flat Mittag--Leffler module. Then there exists a \textsf{Galois
	connection}:
	\begin{equation}\label{eq:galois-connection}
	    \Sub_{\Alg}(A)\,\galois{\psi}{\phi}\,\qquot(H)
	\end{equation}
	where \(\phi(Q):=A^{co\,Q}\) and \(\psi\) is given by the following
	formula: \(\psi(S)=\bigvee\{Q\in\qquot(H): S\subseteq A^{co\,Q}\}\).
\end{theorem}
If \(B\) is commutative and \(H\) is projective over the base ring then
\(A/B\) is projective \((B\otimes H)\)-Hopf Galois extension
by~\cite[Thm~1.7]{hk-mt:hopf-algebras-and-galois-extensions} and thus the
above theorem applies as well as~\cite[Prop.~3.2]{ps:gal-cor-hopf-bigal}.
\begin{proof}
	We shall show that $\phi$ reflects all suprema: \(A^{co\,\bigvee_{i\in
		    I}\, Q_i}=\bigcap_{i\in I}\,A^{co\,Q_i}\). From the set
	of inequalities: \(\bigvee_{i\in I}\,Q_i\geq Q_j\ (\forall_{j\in I})\)
	it follows that \(A^{co\,\bigvee_{i\in I}\, Q_i}\subseteq\bigcap_{i\in
		I}\,A^{co\,Q_i}\).  Let us fix an element $a\in\bigcap_{i\in
	    I}\,A^{co\,Q_i}$. We let $I_i$ denote the coideal and right ideal
	such that $Q_i=H/I_i$. We identify $A\otimes I_i$ with a submodule of
	$A\otimes H$, what can be done under the assumption that $A$ is flat
	over \(R\). We want to show that:
	\[\forall_{i\in I}\;a\in A^{co\,Q_i}\ \Leftrightarrow\ \forall_{i\in I}\;\delta(a)-a\otimes 1\in A\otimes I_i\ \Leftrightarrow\ \delta(a)-a\otimes 1\in A\otimes\bigcap_{i\in I}\,I_i\ \Leftrightarrow\ a\in A^{co\,\bigvee_{i\in I} Q_i}\]
	The first equivalence is clear, the second follows from the equality:
	$\bigcap_{i\in I}A\otimes I_i=A\otimes \bigcap_{i\in I} I_i$ which
	holds since flat Mittag--Leffler modules have the intersection
	property. It remains to show that if $\delta(a)-a\otimes 1\in
	A\otimes\bigcap_{i\in I}I_i$ then \mbox{$\delta(a)-a\otimes 1\in
	    A\otimes\bigwedge_{i\in I}I_i$.} Then it follows that
	\(\delta(a)-a\otimes 1\in A\otimes\bigcap_{i\in I}I_i\Leftrightarrow
	    a\in A^{co\,\vee_{i\in I}Q_i}\). We proceed in three steps: we
	first prove this for \(H\), then for \(A\otimes H\) and finally for
	a general \(H\)-comodule algebra \(A\).  For $A=H$ this follows from
	existence of the Galois connection:
    \begin{equation*}
	\qsub(H)\lgalois{K\elmap{}H/HK^+}{\hspace{-.1cm}H^{co\,Q}\ellmap{}Q}\qquot(H)
    \end{equation*}
    where \(\qsub(H)\) is the lattice of right coideal subalgebras of \(H\).
    Note that the infimum in both \(\qsub(H)\) and \(\Sub_\Alg(H)\) is given
    by intersection thus the above adjunction extends to
    \(\Sub_\Alg(H)\sgalois{}{}\qquot(H)\). Now the proof for \(A\otimes H\):
    let \(x=\sum_{k=1}^na_k\otimes h_k\in A\otimes H\) be such that
    \(\sum_{k=1}^na_k\otimes\Delta(h_k)-\sum_{k=1}^na_k\otimes h_k\otimes
	1_H\in A\otimes H\otimes \bigcap_{i\in I}I_i\). Since \(A\) is flat
    Mittag--Leffler module, every finitely generated submodule of \(A\) is
    contained in a projective submodule~\cite[Thm~2.9]{dh-jt:mittag-leffler}.
    Choose a dual basis for the projective submodule \(A_0\) which contains
    all the \(\{a_k\}\): \(\{e_j, e^j\}_{j\in J}\), where \(e_j\in A_0\) and
    \(e^j\in A_0^*\).  Then for every \(j\in J\) we have:
    \[\sum_k e^j(a_k)\Delta(h_k)-e^j(a_k)h_k\otimes1\in H\otimes\bigcap_{i\in I}I_i,\text{ thus } \sum_k e^j(a_k)\Delta(h_k)-e^j(a_k)h_k\otimes1\in H\otimes\bigwedge_{i\in I}I_i.\]
    It follows that \(id_A\otimes\Delta(x)-x\otimes 1_H\in A\otimes
	H\otimes\bigwedge_{i\in I}I_i\). For general case, observe that if
    $a_{\mathit{(0)}}\otimes a_{\mathit{(1)}}-a\otimes 1\in
    A\otimes\bigcap_{i\in I}I_i$ then $a_{\mathit{(0)}}\otimes
    a_{\mathit{(1)}}\otimes a_{\mathit{(2)}}-a_{\mathit{(0)}}\otimes
    a_{\mathit{(1)}}\otimes 1\in A\otimes H\otimes\bigcap_{i\in I}I_i$, thus
    by the previous case $a_{\mathit{(0)}}\otimes a_{\mathit{(1)}}\otimes
    a_{\mathit{(2)}}-a_{\mathit{(0)}}\otimes a_{\mathit{(1)}}\otimes 1\in
    A\otimes H\otimes\bigwedge_{i\in I}I_i$.  Computing
    $\id_A\otimes\epsilon\otimes\id_H$ we get $\delta(a)-a\otimes 1\in
    A\otimes\bigwedge_{i\in I}I_i$. The formula
    \(\psi(S)=\bigvee\{Q\in\qquot(H): S\subseteq A^{co\,Q}\}\) is an easy
    consequence of the Galois connection properties.
\end{proof}

\subsection{Closed Elements}
The main aim of this section is to characterise the closed elements
of the correspondence~\eqref{eq:galois-connection}.
\begin{proposition}\label{prop:mono}
	Let $A$ be an $H$-comodule algebra over a ring $R$ with surjective
	canonical map and let $A$ be a $Q_1$-Galois and a $Q_2$-Galois
	extension where $Q_1,Q_2\in \qquot(H)$. Then:
	\[A^{co\,Q_1}=A^{co\,Q_2}\ \Rightarrow\ Q_1=Q_2\]
\end{proposition}
\begin{proof}
Let $B=A^{co\,Q_1}=A^{co\,Q_2}$ then we have the following commutative diagram:
\begin{center}
\begin{tikzpicture}[>=angle 60,thick]
\matrix[matrix,column sep=7mm,row sep=7mm]{
   &                             & \node(AQ1) {\(A\otimes Q_1\)};\\
   \node(A) {\(A\otimes_B A\)}; & \node(B) {\(A\otimes_{A^{co\,H}}A\)}; & \node(C)   {\(A\otimes H\)};\\
   &                             & \node(AQ2) {\(A\otimes Q_2\)};\\
};
\begin{scope}
\draw[->] (A) -- node[above]{$\can_{Q_1}$} (AQ1);
\draw[->] (A) -- node[below]{$\can_{Q_2}$} (AQ2);
\draw[<<-] (A) -- (B);
\draw[->>] (B) -- node[above,pos=.4]{$\can$}(C);
\draw[->] (C) -- node[right]{$id\otimes\pi_1$}(AQ1);
\draw[->] (C) -- node[right]{$id\otimes\pi_2$}(AQ2);
\end{scope}
\end{tikzpicture}
\end{center}
The maps $\can_{Q_1}$ and $\can_{Q_2}$ are isomorphisms. Let
$f:=(\can_{Q_1}\circ\can_{Q_2}^{-1})\circ(id\otimes\pi_2)$. By commutativity
of the above diagram, $f\circ\can$ and $(id\otimes\pi_1)\circ\can$ are equal.
Moreover, surjectivity of $\can$ yields the equality $(\can_{Q_1}\circ
\can_{Q_2}^{-1})\circ(id\otimes\pi_2)=(id\otimes\pi_1)$. It follows that there
exists $\pi:\,Q_1\ir Q_2$ such that $\can_1\circ \can_2^{-1}=id\otimes \pi$
and $\pi\circ\pi_2=\pi_1$.  Furthermore, $\pi$ is right $H$-linear and
\(H\)-colinear, thus $Q_2\succcurlyeq Q_1$.  In the same way we obtain that
$Q_1\succcurlyeq Q_2$.  Using antisymmetry of \(\succcurlyeq\) we get
$Q_1=Q_2$.
\end{proof}
\begin{corollary}\label{cor:Q-Galois_closed}
    Let \(A\) be an \(H\)-comodule algebra with epimorphic canonical map
    \(\can_H\) such that the Galois connection~\eqref{eq:galois-connection}
    exists. Then \(Q\in \qquot(H)\) is a \textsf{closed element} of 
    Galois connection~\eqref{eq:galois-connection} if \(A/A^{co\,Q}\) is
    \(Q\)-Galois.
\end{corollary}
\begin{proof}
Fix $A^{co Q}$ for some $Q\in \qquot(H)$ then $\phi^{-1}(A^{co\,Q})$ is an
upper-sublattice  of $\qquot(H)$ (i.e. it is a subposet closed under finite
suprema) which has the greatest element, namely $\widetilde
Q=\psi(A^{co\,Q})$. Moreover, $\widetilde Q$ is the only closed element
belonging to $\phi^{-1}(A^{co\,Q})$. Both $Q\leq\psi(A^{co\,Q})$ and the
assumption that $A/A^{co\,Q}$ is $Q$-Galois imply that \(A/A^{co\,\widetilde
	Q}\) is \(\widetilde Q\)-Galois. To this end, we consider the
commutative diagram:
\begin{center}
\begin{tikzpicture}
    \node (A) at (0,0)	{\(A\otimes_B A\)};
    \node (B) at (3cm,0)	{\(A\otimes H\)};
    \node (C) at (0cm,-1.7cm) {\(A\otimes_{A^{co\,\widetilde Q}}A\)};
    \node (D) at (3cm,-1.7cm) {\(A\otimes \widetilde Q\)};
    \node (E) at (0cm,-3.4cm) {\(A\otimes_{A^{co\,Q}}A\)};
    \node (F) at (3cm,-3.4cm) {\(A\otimes Q\)};
    \begin{scope}[>=angle 60,thick]
	\draw[->>] (A) -- node[above]{\(\can_H\)} (B);
	\draw[->>] (A) -- (C);
	\draw[->>] (B) -- (D);
	\draw[->]  (C) -- node[below]{\(\can_{\widetilde Q}\)} (D);
	\draw[->]  (E) -- node[below]{\(\can_Q\)} node [above]{\(\simeq\)}  (F);
	\draw[->]  (C) -- node[left]{\(=\)}(E);
	\draw[->>] (D) -- (F);
    \end{scope}
\end{tikzpicture}
\end{center}
From the lower commutative square we get that $\can_{\widetilde Q}$ is a
monomorphism and from the upper commutative square we deduce that
$\can_{\widetilde Q}$ is onto. Unless $\widetilde Q=Q$ we get a contradiction
with the previous proposition.
\end{proof}
The above result applies also to the Galois correspondence
of~\citet{ps:gal-cor-hopf-bigal}, as it is the same as the Galois connection
of Theorem~\ref{thm:existence} (see
Proposition~\ref{prop:properties-of-adjunction}(3)). Since for a finite
dimensional Hopf algebra \(H\) for every \(Q\) the extension \(A/A^{co\,Q}\)
is \(Q\)-Galois (see~\cite[Cor.~3.3]{ps-hs:gen-hopf-galois}) we get the
following statement.
\begin{proposition}\label{prop:finite_case}
    Let \(H\) be a finite dimensional Hopf algebra over a field \(k\). Let
    \(A/B\) be an \(H\)-Hopf Galois extension. Then every \(Q\in\qquot(H)\) is
    closed. 
\end{proposition}

Now, we adopt~\cite[Def.~3.3]{ps:gal-cor-hopf-bigal} to our setting.
\begin{definition}\label{defi:admissible_II}
    Let \(C\) be a \(R\)-coalgebra and let \(C\smpr{}\widetilde{C}\) be
    a coalgebra quotient. Then \(\widetilde{C}\) is called left (right)
    admissible if it is \(R\)-flat (hence faithfully flat) and \(C\) is
    left (right) faithfully coflat over~\(\widetilde{C}\). 

    Let \(S\) belong to \(\Sub_\Alg(A/B)\) for an \(H\)-extension \(A/B\).
    Then~\(S\) is called right (left) admissible if:
    \begin{enumerate}[topsep=0pt,noitemsep]
	\item \(A\) is right (left) faithfully flat over~\(S\),
	\item for right admissibility, the composition: 
	    \begin{align*}
	    \can_S:                 & A\otimes_SA\sir A\otimes_{A^{co\,\psi(S)}}A\mpr{\can_{\psi(S)}}A\otimes\psi(S),
	    \intertext{is a bijection, while for left admissibility the map:}
	    \can_{S^{\mathit{op}}}: & A^{\mathit{op}}\otimes_{S^{\mathit{op}}}A^{\mathit{op}}\sir A^{\mathit{op}}\otimes_{(A^{\mathit{op}})^{co\,\psi(S^{\mathit{op}})}}A^{\mathit{op}}\mpr{\can_{\psi(S^{\mathit{op}})}}A^{\mathit{op}}\otimes\psi(S^{\mathit{op}}),
	    \end{align*}
	    is a bijection. These maps are well defined since \(S\subseteq
		A^{co\,\psi(S)}\). 
	\item \(\psi(S)\) is flat over~\(R\).
    \end{enumerate}
    An element is called admissible if it is both left and right admissible.
\end{definition}
Note that the definition is symmetric in the sense that \(S\subseteq A\) is
right (left) admissible if and only if \(S^{\mathit{op}}\subseteq
A^{\mathit{op}}\) is left (right) admissible.
\begin{remark}\label{rem:faithfully_flatness}
    Let \(A/B\) be an \(H\)-extension, such that the Galois
    correspondence~\eqref{eq:galois-connection} exists. Let
    \(S\in\Sub_\Alg(A/A^{co\,H})\). Then the following holds:
    \begin{enumerate}[noitemsep,nolistsep]
	\item[(i)] if \(\can_S:A\otimes_SA\sir A\otimes\psi(S)\) is an
	    isomorphism and \(A\) is right or left is faithfully flat
	    over~\(S\) then~\(S\) is a closed element
	    of~\eqref{eq:galois-connection};
	\item[(ii)] if the natural projection \(A\otimes_SA\eir
		A\otimes_{A^{co\,\psi(S)}}A\) is a bijection and \(A\) is
	    right or left faithfully flat over \(S\) then
	    \(S=A^{co\,\psi(S)}\), i.e. \(S\) is closed element of
	    \(\Sub_\Alg(A/B)\) in~\eqref{eq:galois-connection}.
    \end{enumerate}
\end{remark}
\begin{proof}
    First let us prove (2). For this let us consider the commutative diagram:
    \begin{center}
	\begin{tikzpicture}
	    \matrix[column sep=1cm,row sep=0.7cm]{
		\node (A) {\(S\)}; & \node (B) {\(A\)}; & \node (C) {\(A\otimes_SA\)};\\
		\node (D) {\(A^{co\,\psi(S)}\)}; & \node (E) {\(A\)}; & \node (F) {\(A\otimes_{A^{co\,\psi(S)}}A\)};\\
	    };
	    \draw[->,dashed] (D) -- (A);
	    \node[yshift=-7mm,rotate=90] at (B) {\(=\)};
	    \draw[->] (C) --node[above,rotate=-90]{\(\simeq\)} (F);
	    \node[xshift=10mm] at (A) {\(\subseteq\)};
	    \node[xshift=10mm] at (D) {\(\subseteq\)};
	    \draw[->] ($(B)+(2mm,1mm)$) -- ($(C)+(-7mm,1mm)$);
	    \draw[->] ($(B)+(2mm,-1mm)$) -- ($(C)+(-7mm,-1mm)$);
	    \draw[->] ($(E)+(2mm,1mm)$) -- ($(F)+(-11mm,1mm)$);
	    \draw[->] ($(E)+(2mm,-1mm)$) -- ($(F)+(-11mm,-1mm)$);
	\end{tikzpicture}
    \end{center}
    The maps \(A\sir A\otimes_SA\) and \(A\sir A\otimes_{A^{co\,\psi(S)}}A\)
    send \(a\in A\) to \(a\otimes 1_A\) or \(1_A\otimes a\) in appropriate
    tensor product.  Since the diagram commutes and the upper row is an
    equaliser, by faithfully flat descent, the dashed arrow exists, i.e.
    \(A^{co\,\psi(S)}\subseteq S\). We get the equality since, \(S\subseteq
	A^{co\,\psi(S)}\) holds by the Galois connection property.

    The first claim follows from (2) when applied to \(S\subseteq
	A^{co\,\psi(S)}\). It was first observed 
    in~\cite[Rem.~1.2]{hs:normal-bases}.
\end{proof}
The Remark~\ref{rem:faithfully_flatness}(ii) for division algebras also follows
from predual Jacobson-Bourbaki
correspondence~\cite[Thm~2.1]{ms:predual-jacobson-bourbaki}.

We introduce the following notation: for an algebra \(A\), \(A^{\mathit{op}}\)
denotes the opposite algebra which multiplication is given by
\(m^{\mathit{op}}(a\otimes b)=ba\).  If \(H\) is a Hopf algebra with bijective
antipode \(S_H\) then the algebra \(H^{\mathit{op}}\) is a Hopf algebra with
the same comultiplication and the antipode \(S_{H^{\mathit{op}}}=S_H^{-1}\).
For a generalised quotient \(Q=H/I\in\qquot(H)\) we put
\(Q^{\mathit{op}}\coloneq H^{\mathit{op}}/S_{H}(I)\in\qquot(H^{\mathit{op}})\)
and we also write \(I^{\mathit{op}}\coloneq S_H(I)\).
\begin{theorem}\label{thm:Galois_closedness}
    Let \(H\) be a Hopf algebra with bijective antipode.  Let \(A/B\) be an
    \(H\)-extension such that \(A/B\) is \(H\)-Galois,
    \(A^{\mathit{op}}/B^{\mathit{op}}\)  is \(H^{\mathit{op}}\)-Galois and
    \(A\) is faithfully flat as both left and right \(B\)-module. Let us
    assume that the Galois connection~\eqref{eq:galois-connection} exists.
    Furthermore, let \(A\) be faithfully flat over~\(R\). Then the Galois
    connection~\eqref{eq:galois-connection} gives rise to a bijection between
    (left, right) admissible objects (thus (left, right) admissible objects
    are closed). 
\end{theorem}
Note that if \(B=A^{co\,H}\) is contained in the center of \(A\) (as it is
assumed in~\cite{ps:gal-cor-hopf-bigal}) then if \(A/B\) is \(H\)-Galois then
\(A^{\mathit{op}}/B^{\mathit{op}}\) is \(H^{\mathit{op}}\)-Galois. We have
a commutative diagram:
\begin{center}
    \begin{tikzpicture}
	\matrix[column sep=1.2cm,row sep=1cm]{
	    \node (A1) {\(A^{\mathit{op}}\otimes_{B^{\mathit{op}}}A^{\mathit{op}}\)}; & \node (B1) {\(A^{\mathit{op}}\otimes H^{\mathit{op}}\)};\\
	    \node (A2) {\(A\otimes_{B}A\)}; & \node (B2) {\(A\otimes H\)};\\
	};
	\draw[->] (A1) --node[above]{\(\can_{B^{\mathit{op}}}\)} (B1); 
	\draw[->] (A2) --node[below]{\(\can_{B}\)} (B2); 
	\draw[->] (A1) --node[left]{\(\tau\)} (A2);
	\draw[->] (B1) --node[right]{\(\alpha\)} (B2);
    \end{tikzpicture}
\end{center}
where \(\tau(a\otimes_{B^{\mathit{op}}} b)=b\otimes_B a\) is well defined
since \(B=B^{\mathit{op}}\) is contained in the center of \(A\) and
\(\alpha(a\otimes h)=a_{(0)}\otimes S_H^{-1}(h)a_{(1)}\) is an isomorphism with
inverse \(\alpha^{-1}(a\otimes h)=a_{(0)}\otimes a_{(1)}S_H(h)\).

The proof of~\cite[Thm~3.6]{ps:gal-cor-hopf-bigal} applies (with minor changes
which we spot below). Here we consider a right \(H\)-comodule algebra,
while \citeauthor{ps:gal-cor-hopf-bigal} considers left \(L(H,A)\)-comodule
algebra structure. The proof relies on~\cite[Rem.~1.2 and Thm
1.4]{hs:normal-bases}. We let \((\phi^{\mathit{op}},\psi^{\mathit{op}})\)
denote the Galois connection~\eqref{eq:galois-connection} for the
\(H^{\mathit{op}}\)-comodule algebra \(A^{\mathit{op}}\). We start with
a basic lemma:
\begin{lemma}\label{lem:galois_and_op}
    Let \(H\) be a Hopf algebra with bijective antipode. Let \(A/B\) be an
    \(H\)-extensions such that the Galois
    connection~\eqref{eq:galois-connection} exists. Then it also exists for
    the \(H^{\mathit{op}}\)-comodule algebra \(A^{\mathit{op}}\). Furthermore,
    \(\phi^{\mathit{op}}(Q)=\bigl(\phi(Q)\bigr)^{\mathit{op}}\) and
    \(\psi(S^{\mathit{op}})=\bigl(\psi(S)\bigr)^{\mathit{op}}\), where
    \(S\in\Sub_\Alg(A/B)\) and \(Q\in\qquot(H)\). 
\end{lemma}
\begin{proof}
    The first equation:
    \(\phi^{\mathit{op}}(Q)=\bigl(\phi(Q)\bigr)^{\mathit{op}}\) is proved
    in~\cite[Prop.~3.5]{ps:gal-cor-hopf-bigal}. 
    Both maps: \(\Sub_\Alg(A/B)\ni
	S\selmap{}S^{\mathit{op}}\in\Sub_\Alg(A^{\mathit{op}}/B^{\mathit{op}})\)
    and \(\qquot(H)\ni Q\selmap{}Q^{\mathit{op}}\in\qquot(H^{\mathit{op}})\)
    are isomorphisms of posets thus \(\phi\) reflacts suprema if and only if
    \(\phi^{\mathit{op}}\) reflacts them.  It remains to show that
    \(\psi^{\mathit{op}}(S^{\mathit{op}})=\bigl(\psi(S)\bigr)^{\mathit{op}}\).
    First, let us observe that for any set \(\mathcal{O}\) of right ideal
    coideals we have \(\bigwedge_{I\in
	    \mathcal{O}}I^{\mathit{op}}=\left(\bigwedge_{I\in
		\mathcal{O}}I\right)^{\mathit{op}}\):
    \begin{equation}\label{eq:wedge_and_op}
	\bigwedge_{I\in \mathcal{O}}I^{\mathit{op}} = \mathop{+}\limits_{\substack{J\in\qid(H^{\mathit{op}})\\ J\subseteq\mathop\cap\limits_{I\in \mathcal{O}}I^{\mathit{op}}}} J
	= \mathop{+}\limits_{\substack{J\in\qid(H)\\J^{\mathit{op}}\subseteq(\mathop\cap\limits_{I\in \mathcal{O}}I)^{\mathit{op}}}}J^{\mathit{op}}
	= \mathop{+}\limits_{\substack{J\in\qid(H)\\J\subseteq\mathop\cap\limits_{I\in \mathcal{O}}I}}J^{\mathit{op}} = \Bigl(\bigwedge_{I\in \mathcal{O}}I\Bigr)^{\mathit{op}}
    \end{equation}
    Let \(\overline{\psi}(S)=\ker(H\sir\psi(S))\).  Now using formula~\eqref{eq:galois-connection} for
    \(\psi\) we get:
    \begin{equation*}
            \overline{\psi}(S)=\bigwedge\{I:\ S\subseteq A^{co\,H/I}\},\quad \overline{\psi}^{\mathit{op}}(S^{\mathit{op}})=\bigwedge\{I^{\mathit{op}}:\ S\subseteq (A^{\mathit{op}})^{co\,H^{\mathit{op}}/I^{\mathit{op}}}\}
    \end{equation*}
    since
    \(\bigl(A^{co\,H/I}\bigr)^{\mathit{op}}=(A^{\mathit{op}})^{co\,H^{\mathit{op}}/I^{\mathit{op}}}\)
    and \(\qquot(H)\ni I\selmap{}I^{\mathit{op}}\coloneq
	S_{H^{\mathit{op}}}(I)\in\qquot(H^{\mathit{op}})\) is a bijection the
    two above sets are in bijective correspondence:
    \(I\selmap{}I^{\mathit{op}}\). Now, the formula
    \(\psi^{\mathit{op}}(S^{\mathit{op}})=\bigl(\psi(S)\bigr)^{\mathit{op}}\)
    follows from equation~\eqref{eq:wedge_and_op}.  
\end{proof}
\begin{proposition}\label{prop:Schauenburg_admissibility}
    Let \(H\) be a Hopf algebra with bijective antipode. Let \(A\) be
    \(H\)-extension of \(B\) such that \(A\) is \(H\)-Galois,
    \(A^{\mathit{op}}\) is \(H^{\mathit{op}}\)-Galois and the Galois
    connection~\eqref{eq:galois-connection} exists, \(_BA\) and \(A_B\) are
    faithfully flat and also it is faithfully flat as an \(R\)-module. Then:
    \begin{enumerate}[noitemsep,nolistsep]
	\item[(1)] if \(S\in\Sub(A/B)\) is right admissible then so is
	    \(\psi(S)\) and \(\phi\psi(S)=S\),
	\item[(2)] if \(Q\in\qquot(H)\) is left admissible then so is
	    \(\phi(Q)\coloneq A^{co\,Q}\) and \(\psi\phi(Q)=Q\). 
    \end{enumerate}
\end{proposition}
\begin{proof}
    We first prove (2): \(\psi(Q)\) is left admissible
    by~\cite[Thm~1.4]{hs:normal-bases} applied to \(A^{\mathit{op}}\). The
    equality
    \(\psi^{\mathit{op}}\phi^{\mathit{op}}(Q^{\mathit{op}})=Q^{\mathit{op}}\)
    follows from Corollary~\ref{cor:Q-Galois_closed}. Using
    Lemma~\ref{lem:galois_and_op} we conclude
    \(\bigl(\psi\phi(Q)\bigr)^{\mathit{op}}=Q^{\mathit{op}}\), and since
    \(\qquot(H)\ni Q\selmap{}Q^{\mathit{op}}\in\qquot(H^{\mathit{op}})\) is
    a bijection we get \(\psi\phi(Q)=Q\). Now, let us assume that
    \(S\in\Sub(A/B)\) is right admissible. Then by
    Remark~\ref{rem:faithfully_flatness}(i) \(S=\phi\psi(S)\) and \(A/S\) is
    \(Q\)-Galois. Since, \(A\) is faithfully flat \(R\)-module and \(A_S\) is
    faithfully flat, it follows from the isomorphism \(A\otimes_SA\cong
	A\otimes\psi(S)\) that \(\psi(S)\) is faithfully flat \(R\)-module. It
    remains to show that \(H\) is faithfully coflat as right
    \(\psi(S)\)-comodule. For this we observe, as
    in~\cite[Prop.~3.4]{ps:gal-cor-hopf-bigal}, that for any left
    \(\psi(S)\)-comodule \(V\) there is an isomorphism:
    \(A\cotensor_H(H\cotensor_{\psi(S)}V)\cong(A\cotensor_H
	H)\cotensor_{\psi(S)}V\cong A\cotensor_{\psi(S)}V\), where the first
    isomorphism exists since \(A\) is coflat right \(H\)-comodule. Now, \(H\)
    is faithfully coflat right \(\psi(S)\) comodule since \(A\) faithfully
    coflat as both right \(H\) and \(\psi(S)\)-comodule (what follows from
    faithfully flatness of \(A\) as right \(B\) and \(S\)-module and the
    canonical isomorphisms: \(\can_B:A\otimes_BA\cong A\otimes H\) and
    \(\can_S:A\otimes_SA\cong A\otimes\psi(S)\)).
\end{proof}

\begin{proposition}\label{prop:Schauenburg_oposite}
    Let \(H\) be a flat \(R\)-Hopf algebra with bijective antipode.  The map
    \(\qquot(H)\ni I\selmap{}I^{\mathit{op}}\coloneq
	S_H(I)\in\qquot(H^{\mathit{op}})\) is a bijection with inverse
    \(\qquot(H^{\mathit{op}})\ni J\selmap{}J^{\mathit{op}}\coloneq
	S_{H^{\mathit{op}}}(J)\in\qquot(H)\).  The coideal \(S_H(I)\) is right
    (left) admissible if and only if \(I\) is right (left) admissible.
\end{proposition}
\begin{proof}
    The map \(\qquot(H)\ni I\selmap{}I^{\mathit{op}}\coloneq
	S_H(I)\in\qquot(H^{\mathit{op}})\) is bijection since the antipode
    \(S_H\) is bijective and the inverse \(S_H^{-1}=S_{H^{\mathit{op}}}\).

    The antipode defines an \(R\)-linear bijection \(H/I\cong H/S_H(I)\) thus
    \(H/S_H(I)\) is \(R\)-flat whenever \(H/I\) is. Let \(V\) be a left
    \(H/S_H(I)\)-comodule. We make it a right \(H/I\)-comodule by \(V\ni
	v\selmap{}v_{(0)}\otimes S_H^{-1}(v_{(1)})\). Then we have 
    \begin{align*}
	H\cotensor_{H/S_H(I)}V &\cong V\cotensor_{H/I}H\\
	\sum_{i} h_i\otimes v_i & \selmap{} \sum_{i} v_i\otimes h_i 
    \end{align*}
    It is natural in \(V\), hence if \(H\) is right faithfully coflat over
    \(H/S_H(I)\). This shows that \(S_H(I)\) is right admissible if and only
    if \(I\) is left admissible. That \(I\) is right admissible if and only if
    \(S_H(I)\) is left admissible is proved in the same way.
\end{proof}

Now we present the proof of Theorem~\ref{thm:Galois_closedness} which is due
to Schauenburg.
\begin{proofof}{Theorem~\ref{thm:Galois_closedness}}
    We let \((\phi^{\mathit{op}},\psi^{\mathit{op}})\) denote the Galois
    connection~\eqref{eq:galois-connection} for \(A^{\mathit{op}}\)
    \(H^{\mathit{op}}\)-comodule algebra instead of
    \(A\) -- an \(H\)-comodule algebra. 

    Let \(S\subseteq A\) be left admissible. Then \(S^{\mathit{op}}\subseteq
	A^{\mathit{op}}\) is right admissible, thus
    \(\psi^{\mathit{op}}(S^{\mathit{op}})\) is right admissible and by
    Proposition~\ref{prop:Schauenburg_admissibility}(1),
    \(\phi^{\mathit{op}}\psi^{\mathit{op}}(S^{\mathit{op}})=S^{\mathit{op}}\).
    It follows that
    \(S=\bigl(\phi^{\mathit{op}}\psi^{\mathit{op}}(S^{\mathit{op}})\bigr)^{\mathit{op}}=\phi\bigl(\psi^{\mathit{op}}(S^{\mathit{op}})^{\mathit{op}}\bigr)\),
    by Lemma~\ref{lem:galois_and_op}. Using
    Proposition~\ref{prop:Schauenburg_oposite} we conclude that
    \(\bigl(\psi^{\mathit{op}}(S^{\mathit{op}})\bigr)^{\mathit{op}}\) is left
    admissible thus:
    \(\psi(S)=\psi\phi\bigl(\psi^{\mathit{op}}(S^{\mathit{op}})^{\mathit{op}}\bigr)=\psi^{\mathit{op}}(S^{\mathit{op}})^{\mathit{op}}\),
    hence \(S=\phi\psi(S)\) and \(\psi(S)\) is left admissible. 

    We follow the argument of Schauenburg. Let \(I\) be right admissible. Then
    \(I^{\mathit{op}}\) is left admissible, so is
    \(\phi^{\mathit{op}}(H^{\mathit{op}}/I^{\mathit{op}})=\phi(H/I)^{\mathit{op}}\),
    thus \(\phi(H/I)\) is right admissible, and
    \(H^{\mathit{op}}/I^{\mathit{op}}=\psi^{\mathit{op}}\phi^{\mathit{op}}(H^{\mathit{op}}/I^{\mathit{op}})=\bigl(\psi\phi(H/I)\bigr)^{\mathit{op}}\)
    hence
    \(I^{\mathit{op}}=\bigl(\overline{\psi}\phi(H/I)\bigr)^{\mathit{op}}\) and
    thus \(I=\overline{\psi}\phi(H/I)\), i.e. \(H/I=\psi\phi(H/I)\).
\end{proofof}
 
\section{Galois theory for Galois coextensions}
In this section we describe the Galois theory for Galois coextensions. We
begin with some basic definitions.
\begin{definition}\label{defi:module-coalgebra}
	Let \(C\) be a coalgebra and \(H\) a Hopf algebra, both over a ring
	\(R\). We call \(C\) an \mybf{\(H\)-module coalgebra} if it is an
	\(H\)-module such that the \(H\)-action \mbox{\(H\otimes C\sir C\)} is
	a coalgebra map:
	\begin{equation*}
	  \Delta_C(h\cdot c)=\Delta_C(h)\Delta_H(c),\quad\epsilon_C(h\cdot c)=\epsilon_C(h)\epsilon_H(c).
	\end{equation*}
	Let \(C^H\coloneq\nicefrac{C}{H^+C}\) denote the \mybf{invariant
	    coalgebra}. We call \(C\sir C^H\) an \(H\)\textsf{-coexten\-sion}.
\end{definition}
\begin{definition}\label{defi:coGalois}
	An \(H\)-module coalgebra \(C\) is called an \mybf{\(H\)-Galois
	    coextension} if the canonical map
	\[\can_H:H\otimes C\sir C\cotensor_{C^H}C,\quad h\otimes c\elmap{}hc_{(1)}\otimes c_{(2)}\]	
	is a bijection, where $C$ is considered as a left and right
	$C^H$-comodule in a standard way. More generally, if
	\(K\in\coid{l}(H)\) is a left coideal then \(K^+C\) is a coideal (at
	least, when the base ring \(R\) is a field) and an \(H\)-submodule of
	\(C\).  The coextension \(C\sir C^K=\nicefrac{C}{K^+C}\) is called
	Galois if the canonical map
	\begin{equation}\label{eq:can_K}
		\can_K:K\otimes C\sir C\cotensor_{C^K}C,\quad k\otimes c\elmap{}kc_{(1)}\otimes c_{(2)}
	\end{equation}
	is a bijection.
\end{definition}
Basic example of an \(H\)-module coalgebra is \(H\) itself.  Then \(H^H=R\)
and \(H\cotensor_{H^H}H=H\otimes H\). The inverse of the canonical map is
given by \(\can_H^{-1}(k\otimes h)=kS\left(h_{(1)}\right)\otimes h_{(2)}\).
\begin{definition}
	Let \(C\) be an \(H\)-module coalgebra. We let \(\Quot(C)=\{C/I:\
		I-\text{ a coideal of }C\}\). It is a complete lattice. We let
	\(\Quot(\nicefrac{C}{C^H})\) denote the interval \(\{Q\in\Quot(C):
		C^H\leq Q\leq C\}\) in \(\Quot(C)\).
\end{definition}
\begin{proposition}\label{prop:cogalois-extensions}
    Let \(C\) be an \(H\)-module coalgebra over a field~\(k\). Then there
    exists a \textsf{Galois connection}:
    \begin{align}\label{eq:galois-for-module-coalgebras}
	\Quot(\nicefrac{C}{C^{H}})\lgalois{}{C/(I+k1_H)^+C\ellmap{}I}\coid{l}(H)
    \end{align}
\end{proposition}
\begin{proof}
	The supremum in \(\coid{l}(H)\) is given be sum of submodules. Thus
	the lattice of left coideals is complete. Furthermore, if \(I\) is
	a right coideal then \(I+k1_H\) is also a right coideal. It is enough
	to show that the map \(\coid{l}(H)\ni I\mapsto I^+C\in\coId(C)\)
	preserves all suprema when we restrict to right coideals which
	contain~\(1_H\). Let \(I_\alpha\in\qid(H)\), for \(\alpha\in\Lambda\),
	then \((+_\alpha I_\alpha)^+=+_\alpha(I_\alpha^+)\).  The non trivial
	inclusion is \((+_\alpha I_\alpha)^+\subseteq+_\alpha(I_\alpha^+)\).
	Let \(k=\sum_\alpha k_\alpha\in(+_\alpha I_\alpha)^+\), i.e.
	\(k_\alpha\in I_\alpha\) and \(\sum_\alpha k_\alpha\in\ker\epsilon\).
	Then \(\sum_\alpha
	    k_\alpha=\sum_\alpha(k_\alpha-\epsilon(k_\alpha)1)+\sum_\alpha(\epsilon(k_\alpha)1)=\sum_\alpha(k_\alpha-\epsilon(k_\alpha)1)\).
	Now each \(k_\alpha-\epsilon(k_\alpha)1\in I_\alpha^+\) and hence
	\(k\in +_\alpha(I_\alpha^+)\).
\end{proof}
\begin{theorem}\label{thm:coGalois_mono}
    Let \(C\) be an \(H\)-module coalgebra over a field \(k\) with monomorphic
    canonical map \(\can_H\). Let \(K_1,K_2\) be two left coideals of \(H\)
    such that both \(\can_{K_1}\) and \(\can_{K_2}\) are bijections.  Then
    \(K_1=K_2\) whenever \(C^{K_1}=C^{K_2}\). 
\end{theorem}
\begin{proof}
    We have the following commutative diagram:
    \begin{center}
	\begin{tikzpicture}[>=angle 60,thick]
	\matrix[matrix,column sep=7mm,row sep=7mm]{
	    &                         & \node(AQ1) {\(K_1\otimes C\)};\\
	    \node(A) {\(C\cotensor_{C^{K_i}} C\)}; & \node(B) {\(C\cotensor_{C^H}C\)}; & \node(C)   {\(H\otimes C\)};\\
	    &                         & \node(AQ2) {\(K_2\otimes C\)};\\
	};
	\begin{scope}
	    \draw[<-] (A) -- node[above]{$\can_{K_1}$} (AQ1);
	    \draw[<-] (A) -- node[below]{$\can_{K_2}$} (AQ2);
	    \draw[>->] (A) -- (B);
	    \draw[<-<] (B) -- node[above,pos=.4]{$\can$}(C);
	    \draw[<-] (C) -- node[right]{$i_1\otimes\id$}(AQ1);
	    \draw[<-] (C) -- node[right]{$i_2\otimes\id $}(AQ2);
	\end{scope}
	\end{tikzpicture} 
    \end{center}
    It follows that \(i_2\otimes\id\circ
	\left(\can_{K_2}\circ\can_{K_1}^{-1}\right)=i_1\otimes\id\) and thus
    \(K_1\subseteq K_2\); similarly \(K_2\subseteq K_1\).
\end{proof}
\begin{corollary}\label{cor:coGalois_closed_elements}
    Let \(C\) be an \(H\)-coextension with monomorphic canonical map
    \(\can_H\). Then~\(K\) - a left coideal of~\(H\) such that \(1_H\in K\) is
    a closed element of Galois
    connection~\eqref{eq:galois-for-module-coalgebras} if \(C\) is
    \(K\)-Galois.
\end{corollary}
\begin{proof}
    Let \(C\) be a \(K\)-Galois coextension, for some left coideal \(K\)of
    \(H\) such that \(1_H\in K\) and let \(\widetilde K\) be the smallest
    closed left coideal such that \(K\subseteq\widetilde K\). Then we have
    the commutative diagram:
    \begin{center}
	\begin{tikzpicture}
	    \node (A) at (0,0) {$H\otimes C$};
	    \node (B) at (3cm,0) {$C\cotensor_{C^H} C$};
	    \node (C) at (0cm,-1.7cm) {$\widetilde K\otimes C$};
	    \node (D) at (3cm,-1.7cm) {$C\cotensor_{C^{\widetilde K}}C$};
	    \node (E) at (0cm,-3.4cm) {$K\otimes C$};
	    \node (F) at (3cm,-3.4cm) {$C\cotensor_{C^{K}}C$};
	    \begin{scope}[>=angle 60,thick]
		\draw[>->] (A) -- node[above]{$\can_H$} (B);
		\draw[>->] (C) -- (A);
		\draw[>->] (D) -- (B);
		\draw[->]  (C) -- node[above]{$\can_{\widetilde K}$} (D);
		\draw[->]  (E) -- node[below]{$\simeq$} node [above]{$\can_K$}  (F);
		\draw[>->] (E) --(C);
		\draw[->] (F) --node[below,rotate=90]{$=$} (D);
	    \end{scope}
	\end{tikzpicture}
    \end{center}
    From lower commutative square it follows that \(\can_{\widetilde K}\) is
    onto, while from the upper that it is a monomorphism. The result follows
    now from the previous theorem.
\end{proof}

\section[The Hopf algebra case]{Correspondence between left coideal
    subalgebras and generalised quotients} 
We show a new simple prove of Takeuchi correspondence between left coideal
subalgebras and right \(H\)-module coalgebra quotients for a finite
dimensional Hopf algebra. We also show that for arbitrary Hopf algebra \(H\)
the generalised quotient \(Q\) is closed if and only if \(H^{co Q}\subseteq
    H\) is \(Q\)-Galois. Similarly for a left coideal subalgebras: it is
closed if and only if \(H\sir H^K\) is a \(K\)-Galois coextension. 
\begin{theorem}\label{thm:newTakeuchi}
    Let $H$ be a flat Hopf algebra over a ring $R$ with bijective antipode.
    The extension $H/R$ is $H$-Galois and there exists a \textsf{Galois
	connection}:
	\begin{equation}\label{eq:galois-for-hopf-alg}
	    \begin{array}{ccc}
		\bigg\{K\subseteq H:\,K\,\hyphen\text{left coideal subalgebra}\bigg\}&\hspace{-.3cm}\galois{\psi}{\phi}&\hspace{-.3cm}\bigg\{H/I:\,I\,\hyphen\text{right ideal coideal}\bigg\}\\[-2mm]
		=:\qsub(H)&&=:\qquot(H)
	    \end{array}
	\end{equation}
	where $\phi(Q)=H^{co\,Q},\psi(K)=H/K^+H$ is a Galois connection which
	is a restriction of the Galois connection~\eqref{eq:galois-connection}.
	Moreover, this Galois correspondence restricts to normal Hopf
	subalgebras and conormal Hopf quotients.
The following holds:
\begin{enumerate}[topsep=0pt,noitemsep]
    \item[(1)] $K\in\qsub(H)$ such, that $H$ is (left, right) faithfully flat
	over $K$, is a closed element of the above Galois connection,
	\item[(2)] $Q\in\qquot(H)$ such, that $H$ is (left, right) faithfully
	    coflat over $Q$ is a closed element of Galois
	    connection~\eqref{eq:galois-for-hopf-alg},
	\item[(3)] if $H$ is finite dimensional then $\phi$ and $\psi$ are
	    \textsf{inverse bijections}.
\end{enumerate}
The Galois correspondence restricts to bijection between elements
satisfying~(1) and~(2).
\end{theorem}
The existence of this Galois correspondence is proved
in~\citep{ps:gal-cor-hopf-bigal}, where the author notes that it is a folklore
even in the case when \(R\) is a ring rather than just a field. Let us note
that existence does not require \(H\) to have a bijective antipode. The points
(1) and (2) follows from Theorem~\ref{thm:Galois_closedness} (due to
Schauenburg, see also~\cite[Thm~3.10]{ps:gal-cor-hopf-bigal}), while point (3)
follows from~\cite[Thm~6.1]{ss:projectivity-over-comodule-algebras}, where it
is shown that if \(H\) is finite dimensional then it is free over every its
right (or left) coideal subalgebra
(see~\cite[Thm~6.6]{ss:projectivity-over-comodule-algebras} and
also~\cite[Cor.~3.3]{ps-hs:gen-hopf-galois}). This theorem has a long history.
The study of this correspondence, with Hopf algebraic method, goes back
to~\cite{mt:correspondence,mt:rel-hopf-mod}. Then Masuoka proved~(1) and~(2)
for Hopf algebras over a field (with bijective antipode). When the base ring is
a field, \citet[Thm~1.4]{hs:exact-seq-qg} proved that this bijection restricts
to normal Hopf subalgebras and normal Hopf algebra quotients. For Hopf
algebras over more general rings it was shown
by~\citet[Thm~3.10]{ps:gal-cor-hopf-bigal}. We can present a new simple proof
of~\ref{thm:newTakeuchi}(3), which avoids Skryabin result. 
\begin{proofof}{Theorem~\ref{thm:newTakeuchi}(3)} Whenever \(H\) is finite
    dimensional, for every \(Q\) the extension \(H^{co Q}\subseteq H\) is
    \(Q\)-Galois by~\cite[Cor.~3.3]{ps-hs:gen-hopf-galois}. Using
    Proposition~\ref{prop:mono} we get that the map \(\phi\) is
    a monomorphism. To show that it is an isomorphism it is enough to prove
    that \(\psi\) is a monomorphism.  We now want to consider \(H^*\).  To
    distinguish \(\phi\) and \(\psi\) for \(H\) and \(H^*\) we will write
    \(\phi_H\) and \(\psi_H\) considering~\eqref{eq:galois-for-hopf-alg} for
    \(H\) and \(\phi_{H^*}\) and \(\psi_{H^*}\) considering \(H^*\). It turns
    out that \((\psi_H(K))^*=\phi_{H^*}(K^*)\). It can be easily shown that
    \(\can_{K^*}=(\can_K)^*:H^*\otimes_{{H^*}^{co\,K^*}}H^*\sir H^*\otimes
	K^*\) under some natural identifications. Because \(H^*\) is finite
    dimensional as well it follows that \(\can_{K^*}\)  is an isomorphism,
    hence \(\can_K\) is a bijection for every right coideal subalgebra \(K\)
    of \(H\). Now the result follows from Theorem~\ref{thm:coGalois_mono} and
    Proposition~\ref{prop:properties-of-adjunction}.
\end{proofof}
\begin{theorem}\label{thm:closed-of-qquot}
    Let $H$ be a flat Hopf algebra over a ring \(R\). Then 
    \begin{enumerate}
	\item[(i)] \(Q\in\qquot(H)\) is a \textsf{closed element} of Galois
	    connection~\eqref{eq:galois-for-hopf-alg} if and only if
	    \(H/H^{co\,Q}\) is a \(Q\)-Galois extension,
	\item[(ii)] \(K\in\qsub(H)\) is a \textsf{closed element} of the
	    Galois connection~\eqref{eq:galois-for-hopf-alg} if and only if
	    \(H\sir H^K\) is a \(K\)-Galois coextension. 
    \end{enumerate}
\end{theorem}
Note that we do not assume that the antipode of~\(H\) is bijective as it is
done in Theorem~\ref{thm:newTakeuchi}. The flatness of \(H\) is only
needed to show that if \(K\) is closed then \(H\sir H^K\) is a \(K\)-Galois
coextension.
\begin{proof}
    For the first part, it is enough to show that if $Q$ is closed then $H^{co
	Q}\subseteq H$ is a $Q$-Galois (see
    Corollary~\ref{cor:Q-Galois_closed}). If $Q$ is closed then $Q=H/(H^{co
	Q})^+H$. One can show that for any $K\in\qsub(H)$ the following map is
    an isomorphism:
    \begin{equation}\label{eq:canK}
	H\otimes_{K}H\mpr{}H\otimes H/K^+H,\quad h\otimes_K h'\elmap{}hh'_{\mathit{(1)}}\otimes\overline{h'_{\mathit{(2)}}}
    \end{equation}
    Its inverse is given by $H\otimes H/K^+H\ni h\otimes
    \overline{h'}\,\elmap{}\,hS(h'_{\mathit{(1)}})\otimes_K
    h'_{\mathit{(2)}}\in H\otimes_K H$ which is well defined since \(K\) is
    a left coideal. Plugging $K=H^{co\,Q}$ to equation~\eqref{eq:canK} we
    observe that this map is the canonical map~\eqref{eq:canonical-map}
    associated to \(Q\).

    Now, if \(H\sir H/K^+H\) is a \(K\)-Galois coextension then it follows
    from Theorem~\ref{thm:coGalois_mono}, using the same argument as in
    Corollary~\ref{cor:coGalois_closed_elements}, that \(K\) is a closed
    element. Now, let us assume that \(K\) is closed. Then \(K=H^{co\,Q}\) for
    \(Q=H/K^+H\). By~\cite[Thm~1.4(1)]{hs:normal-bases} we have an isomorphism:
    \begin{equation}\label{eq:cocanK}
	H^{co\,Q}\otimes H\ir H\cotensor_QH\quad k\otimes h\elmap{}kh_{(1)}\otimes h_{(2)}
    \end{equation}
    with inverse \(H\cotensor_QH\ni k\otimes h\selmap{}kS(h_{(1)})\otimes
	h_{(2)}\in H^{co\,Q}\otimes H\). The above map is canonical
    map~\eqref{eq:can_K} since \(K=H^{co\,Q}\) and \(Q=H/K^+H\) .
\end{proof}
The above theorem gives an answer to the question when the
bijection~\eqref{eq:galois-for-hopf-alg} holds without extra assumptions.
\begin{corollary}
    The bijective correspondence~\eqref{eq:galois-for-hopf-alg} holds without
    flatness/coflatness assumptions if and only if for every \(Q\in\qquot(H)\)
    \(H/H^{co\,Q}\) is \(Q\)-Galois extension and for every \(K\in\qsub(H)\)
    \(H/H^K\) is a \(K\)-Galois coextension.
\end{corollary}

\section{Crossed product extensions}
Next we describe closed elements of the Galois
connection~\eqref{eq:galois-connection} when \(A\) is a crossed product. An
\(H\)-extension \(A/B\) (over a ring \(R\)) is called \textsf{cleft} if there
exists a convolution invertible \(H\)-comodule map \(\gamma:H\sir A\).  An
\(H\)-extension \(B\subseteq A\) has the \textsf{normal basis property} if and
only if $A$ is isomorphic to $B\otimes_kH$ as a left $B$-module and right
$H$-comodule. If the base ring~\(R\) is a field then the following conditions
are equivalent:
\begin{enumerate}[topsep=0pt,noitemsep]
    \item[(i)] \(A/B\) is a \textsf{cleft} extension,
    \item[(ii)] \(A/B\) is a Hopf--Galois extension with normal basis
	property,
    \item[(iii)] \(A\) is a crossed product of \(B\) and \(H\), i.e.  there
	exists an invertible cocycle \(\sigma:H^{\otimes2}\sir A\) such that
	\(A\cong B\#_\sigma H\) as \(H\)-comodule algebras.
\end{enumerate}
For the proof see~\cite{yd-mt:cleft-comodule-algebras}
and~\cite{rb-sm:crossed-products}. \noindent We refer to~\citep{sm:hopf-alg}
for the theory of crossed products. Let us recall that the underlying
\(R\)-module of \(B\#_\sigma H\) is \(B\otimes H\) while the multiplication is
given by the formula:
\[a\#_\sigma h\cdot b\#_\sigma k:= a\left(h_{(1)}\cdot b\right)\sigma\left(h_{(2)}\otimes k_{(1)}\right)\#_\sigma h_{(3)}k_{(2)}\]
Note that \(H\) acts on \(B\) in a compatible way (\(H\) measures \(B\),
see~\cite[Def.~7.1.1]{sm:hopf-alg}). Let us note that the results presented
below apply to finite Hopf--Galois extensions of division rings, since they are
always crossed products by~\cite[Thm~8.3.7]{sm:hopf-alg}.
\begin{theorem}\label{thm:cleft-case}
    Let \(A/B\) be an \(H\)-crossed product over a ring~\(R\) with \(B\)
    a flat Mittag-Leffler \(R\)-module and \(H\) a flat \(R\)-module. Then the
    Galois correspondence~\eqref{eq:galois-connection} exists. Moreover, an
    element \(Q\in\qquot(H)\) is \textsf{closed} if and only if the extension
    \(A/A^{co Q}\) is \textsf{\(Q\)-Galois}.
\end{theorem}
\begin{proof}
    First of all, the Galois connection~\eqref{eq:galois-connection} exists,
    since we have a diagram:
    \begin{center}
	\begin{tikzpicture}
	    \matrix[column sep=1cm, row sep=1cm]{
		\node (C) {\(\Sub_\Alg(B\#_\sigma H)\)}; & \node (A) {\(\qsub(H)\)};  & \node (B) {\(\qquot(H)\)};\\
	    };
	    \draw[->] ($(A)+(9mm,1.5mm)$) --node[above]{\(\psi\)} ($(B) + (-10mm,1.5mm)$);
	    \draw[<-] ($(A)+(9mm,0mm)$) --node[below]{\(\phi\)} ($(B) + (-10mm,0mm)$);
	    \draw[->] ($(A)+(-9mm,0mm)$) --node[below]{\(\zeta\)} ($(C) + (12.5mm,0mm)$);
	    \draw[<-,dashed] ($(A)+(-9mm,1.5mm)$) --node[above]{\(\omega\)} ($(C) + (12.5mm,1.5mm)$);
	\end{tikzpicture}
    \end{center}
    where \((\phi,\psi)\) is the Galois
    connection~\eqref{eq:galois-for-hopf-alg} and \(\zeta(S)=B\otimes S\). The
    map \(\zeta\) preserves all intersections and so it has a left adjoint
    \(\omega\). Then the Galois connection~\eqref{eq:galois-connection} exists
    and it has the form \((\zeta\circ\phi,\psi\circ \omega)\), since by
    flatness of~\(B\) we have \(B\otimes H^{co\,Q}=(B\otimes H)^{co\,Q}\). 

    Let \(Q\in\qquot(H)\) be a closed element of this Galois connection, i.e.
    \(Q=\psi\omega\zeta\phi(Q)\). Then it is closed element
    of~\eqref{eq:galois-for-hopf-alg}, since it belongs to the image of
    \(\psi\), and thus, by Theorem~\ref{thm:closed-of-qquot}(i), the map:
    \(\can'_{Q}:H\otimes_{H^{co\,Q}}H\sir H\otimes Q\) is an isomorphism.
    We have a commutative diagram:
    \begin{center}
	\begin{tikzpicture}
	    \matrix[column sep=1.4cm,row sep=7mm]{
		\node (A0) {\(A\otimes_{A^{co\,Q}}A\)}; & \node (A1) {\(A\otimes Q\)}; \\
		\node (B0) {\(B\otimes H\otimes_{H^{co\, Q}}H\)}; & \node (B1) {\(B\otimes H\otimes Q\)}; \\
	    };
	    \draw[->] (A0) --node[above]{\(\can_Q\)} (A1);
	    \draw[->] (B0) --node[below]{\(\id_B\otimes\can'_Q\)}node[above]{\(\simeq\)} (B1);
	    \draw[->] (A0) --node[above,rotate=90]{\(\simeq\)} (B0);
	    \draw[->] (A1) --node[below,rotate=90]{\(=\)} (B1);
	\end{tikzpicture}
    \end{center}
    Thus  \(\can_Q\) is an isomorphism. The converse follows from
    Corollary~\ref{cor:Q-Galois_closed}, since \(\can:A\otimes_BA\sir A\otimes
	H\) is a bijection.
\end{proof}
Now we formulate a criterion for closedness of subextensions of \(A/B\)
which generalises~\ref{thm:closed-of-qquot}(ii).
\begin{theorem}
    Let \(A/B\) be an \(H\)-crossed product extension over a ring \(R\)
    with~\(B\) a faithfully flat Mittag--Leffler module and~\(H\) a flat (thus
    faithfully flat) \(R\)-module. Then a subalgebra \(S\in\Sub_\Alg(A/B)\)
    is a \textsf{closed element} of the Galois
    connection~\eqref{eq:galois-connection} if and only if the canonical map:
    \begin{equation}\label{eq:Schneider_canocial_map}
	\can_S: S\otimes A\ir A\cotensor_{\psi(S)} H,\quad\can_S(a\otimes b)=ab_{(1)}\otimes b_{(2)}
    \end{equation}
    is an \textsf{isomorphism}.
\end{theorem}
\begin{proof}
    First let us note that the map \(\can_S\) is well defined since it is
    a composition of \(S\otimes A\sir A^{co\,\psi(S)}\otimes A\), induced by
    the inclusion \(S\subseteq A^{co\psi(S)}\), with \(A^{co\,\psi(S)}\otimes
	A\sir A\cotensor_{\psi(S)}H\) of \cite[Thm~1.4]{hs:normal-bases}.

    Now let us assume that \(\can_S\) is an isomorphism.  We let
    \(K=H^{co\,\psi(S)}\). Since \(B\) is a flat \(R\)-module we have
    \(\left(B\#_\sigma H\right)^{co\,\psi(S)}=B\#_\sigma K\). The following
    diagram commutes:
    \begin{center}
	\hfill\begin{tikzpicture}
	    \matrix[matrix, column sep=1.5cm, row sep=.75cm]{
		\node[anchor=east] (A) {\(S\otimes_B A\)};
		& \node[anchor=west] (B) {\(A\cotensor_{\psi(S)} H\)};\\
		\node[anchor=east] (C) {\(B\#_\sigma K\otimes H\)};
		& \node[anchor=west] (D)
		{\(B\otimes\left(H\cotensor_{\psi(S)}H\right)\)};\\ };
	    \draw[->] (A) --node[above]{\(\can_S\)} (B); \draw[->] (C)
	    --node[below]{\(\beta\)} (D); \draw[->] (A)
	    --node[left]{\(\alpha\)} ($(A)+(0,-1)$); \draw[->] (B)
	    --node[above,rotate=-90]{\(\simeq\)} ($(B)+(0,-1)$); 
	\end{tikzpicture}
	\hfill\refstepcounter{equation}\raisebox{12mm}{(\theequation)}\label{diag:1}\\
    \end{center}
    where \(\alpha:S\otimes_BA\sir\left(B\#_\sigma K\right)\otimes_B\left(B\#_\sigma
	    H\right)\sir B\#_\sigma K\otimes H\) is given by
    \[\alpha\left(a\otimes_Bb\#_\sigma h\right)=\left(a\cdot b\#_\sigma 1_H\right)\otimes h\quad\text{for }a\in S\subseteq B\#_\sigma K,\ b\#_\sigma h\in A\] 
    It is well defined since \(K\) is a left comodule subalgebra of \(H\).
    The second vertical map \(\left(B\#_\sigma H\right)\cotensor_{\psi(S)}
	H\sir B\otimes\left(H\cotensor_{\psi(S)}H\right)\) is the natural
    isomorphism, since \(B\) is flat over \(R\). The map \(\beta:B\#_\sigma
	K\otimes H\sir B\otimes\left(H\cotensor_{\psi(S)}H\right)\) is defined
    by \(\beta(a\#_\sigma k\otimes h)=\left(a\#_\sigma k\cdot 1_B\#_\sigma
	    h_{(1)}\right)\otimes h_{(2)}\). Note that
    \(\beta=\id_B\otimes\can_K\circ\gamma\) where \(\gamma:\left(B\#_\sigma
	    K\right)\otimes H\sir\left(B\#_\sigma K\right)\otimes H\),
    \(\gamma\left(a\#_\sigma k\otimes h\right)\coloneq a\sigma(k_{(1)},
	h_{(1)})\otimes k_{(2)}\otimes h_{(2)}\) is an isomorphism, since
    \(\sigma\) is convolution invertible, while \(\can_K\) denotes the canonical
    map~\eqref{eq:cocanK}. Furthermore, the map \(\can_K\) is an isomorphism
    by Theorem~\ref{thm:closed-of-qquot}(ii), since \(K:=H^{co\,\psi(S)}\).
    By commutativity of the above diagram it follows that \(\alpha\) is an
    isomorphism. Now, let us consider the commutative diagram:
    \begin{center}
	\begin{tikzpicture}
	    \matrix[matrix, column sep=1.5cm, row sep=.75cm]{
		\node (A) {\(S\otimes_B\left(B\#_\sigma H\right)\)}; & \node (B) {\(B\#_\sigma K\otimes H\)};\\
		\node (C) {\(\left(B\#_\sigma K\right)\otimes_B \left(B\#_\sigma H\right)\)}; & \\
	    };
	    \draw[->] (A) --node[above]{\(\alpha\)} (B);
	    \draw[->] (A) -- (C);
	    \draw[->] (C) --node[below right]{\(\simeq\)} (B);
	\end{tikzpicture}
    \end{center}
    Thus \(S=B\#_\sigma K\) is indeed closed, since \(\alpha\) is an
    isomorphism and \(B\#_\sigma H\) is a faithfully flat \(B\)-module under
    the assumptions made.

    If \(S\) is closed then \(S=A^{co\,\psi(S)}=B\#_\sigma
	(H^{co\,\psi(S)})\), since \(B\) is a flat \(R\)-module. One easily
    checks that \(\alpha\) is an isomorphism with inverse:
    \(\alpha^{-1}(a\#_\sigma k\otimes h)=\left(a\#_\sigma
	    k\right)\otimes_B\left(1_B\#_\sigma h\right)\).  The left coideal
    subalgebra \(K=H^{co\,\psi(S)}\) is a closed element
    of~\eqref{eq:galois-for-hopf-alg} hence by
    Theorem~\ref{thm:closed-of-qquot}(ii) \(\can_K\) is an isomorphism and
    thus \(\beta\) is an isomorphism. It follows from~\eqref{diag:1} that
    \(\can_S\) is an isomorphism as well.
\end{proof}

\end{document}